\documentclass[final,reqno]{amsart}
\usepackage[margin=1.5in,bottom=1.25in]{geometry}	
\usepackage[pagewise]{lineno}
\usepackage{bm}

\usepackage{amsmath}	
\usepackage{amssymb}	
\usepackage{amsfonts}	
\usepackage{amsthm}	
\usepackage[foot]{amsaddr}	
\usepackage{empheq}
\usepackage{breqn}
\usepackage[titletoc,toc]{appendix}

\mathtoolsset{%
}

\usepackage[utf8]{inputenc}	
\usepackage[T1]{fontenc}	

\usepackage[
cal=cm,
]
{mathalfa}

\usepackage{dsfont}	
\newcommand{\BF}{\textbf}

\usepackage[light,scaled=.95]{AlegreyaSans}

\usepackage[proportional,tabular,lining,sf,mono=false]{libertine}

\usepackage{acronym}	

\usepackage[labelfont={bf,small},labelsep=colon,font=small]{caption}	

\usepackage{subcaption}	
\usepackage[svgnames]{xcolor}	
\colorlet{MyRed}{Crimson!60!DarkRed}
\colorlet{MyBlue}{DodgerBlue!75!black}
\colorlet{MyGreen}{DarkGreen}
\colorlet{MyViolet}{DarkMagenta}

\colorlet{MyLightBlue}{DodgerBlue!20}
\colorlet{MyLightGreen}{MyGreen!20}

\colorlet{PrimalColor}{MyBlue}
\colorlet{PrimalFill}{MyLightBlue}
\colorlet{DualColor}{MyRed}

\colorlet{AlertColor}{MyRed}	
\colorlet{BadColor}{MyRed}	
\colorlet{GoodColor}{MyGreen}	
\colorlet{LinkColor}{MediumBlue}	
\colorlet{MacroColor}{MyViolet}
\colorlet{RevColor}{MediumBlue}	

\usepackage{latexsym}	
\usepackage{fontawesome}	
\usepackage{pifont}	

\usepackage{tikz}	
\usetikzlibrary{calc,patterns}	
\usepackage{pgfplots}
\pgfplotsset{compat=1.18}
\usetikzlibrary{patterns, pgfplots.fillbetween}

\usepackage{array}	
\usepackage{booktabs}	
\usepackage[inline,shortlabels]{enumitem}	

\setlist[1]{topsep=\smallskipamount,itemsep=\smallskipamount,left=\parindent}
\setlist[2]{left=0pt}

\usepackage[kerning=true]{microtype}	
\usepackage{ifluatex}
\ifluatex
  \usepackage{microtype}
  \DisableLigatures{encoding = *, family = *}
\else
  \usepackage{microtype}
\fi

\usepackage{tabto}	
\usepackage{xspace}	

\usepackage[sort&compress]{natbib}	

\bibpunct[, ]{[}{]}{,}{}{,}{,}
\setcitestyle{numbers,square}

\usepackage{hyperref}
\hypersetup{
final,
colorlinks=true,
linktocpage=true,
pdfstartview=FitH,
breaklinks=true,
pdfpagemode=UseNone,
pageanchor=true,
pdfpagemode=UseOutlines,
plainpages=false,
bookmarksnumbered,
bookmarksopen=false,
bookmarksopenlevel=1,
hypertexnames=true,
pdfhighlight=/O,
urlcolor=LinkColor,linkcolor=LinkColor,citecolor=LinkColor,	
pdftitle={},
pdfauthor={},
pdfsubject={},
pdfkeywords={},
pdfcreator={pdfLaTeX},
pdfproducer={LaTeX with hyperref}
}

\usepackage[sort&compress,capitalize,nameinlink]{cleveref}	
\crefname{assumption}{Assumption}{Assumptions}
\crefname{case}{Case}{Cases}
\crefname{cond}{Condition}{Conditions}
\crefname{noref}{}{}
\crefname{problem}{Problem}{Problems}

\makeatletter	
\DeclareRobustCommand{\crefnosort}[1]{%
  \begingroup\@cref@sortfalse\cref{#1}\endgroup
}
\makeatother

\usepackage{algpseudocode}	
\usepackage[]{algorithm2e}
\crefname{algocf}{Algorithm}{Algorithms}

\theoremstyle{plain}
\newtheorem{theorem}{Theorem}	
\newtheorem{corollary}{Corollary}	
\newtheorem{lemma}{Lemma}	
\newtheorem{proposition}{Proposition}	

\newtheorem*{theorem*}{Theorem}	
\newtheorem*{corollary*}{Corollary}	

\theoremstyle{definition}
\newtheorem{assumption}{Assumption}	
\newtheorem{definition}{Definition}	
\newtheorem{example}{\raisebox{\depth}{$\blacktriangleright$}~Example}	

\newtheorem*{assumption*}{Assumptions}	
\newtheorem*{definition*}{Definition}	
\newtheorem*{example*}{\raisebox{\depth}{$\blacktriangleright$}~Example}	

\theoremstyle{remark}
\newtheorem{remark}{Remark}	

\newtheorem*{remark*}{Remark}	
\newtheorem*{notation*}{Notation}	

\newcounter{proofstep}

\numberwithin{remark}{section}	
\numberwithin{example}{section}	

\usepackage[showdeletions,suppress]{color-edits}	
\newcommand{\draft}[1]{#1}	

\newcommand{\newmacro}[2]{\newcommand{#1}{\draft{#2}}}	
\newcommand{\newop}[2]{\DeclareMathOperator{#1}{\draft{#2}}}	
\newcommand{\newoplims}[2]{\DeclareMathOperator*{#1}{\draft{#2}}}	
\newcommand{\newsmartmacro}[2]{
	\NewDocumentCommand{#1}{
		E{_}{{}}
	}{
		\draft{#2_{##1}}
	}
}

\newcommand{\eps}{\varepsilon}	
\newcommand{\wilde}{\widetilde}	

\DeclarePairedDelimiterX{\setdef}[2]{\{}{\}}{#1:#2}	
\DeclarePairedDelimiterXPP{\exclude}[1]{\mathopen{}\setminus}{\{}{\}}{}{#1}

\DeclarePairedDelimiterX{\braket}[2]{\langle}{\rangle}{#1,#2}	
\DeclarePairedDelimiterX{\internalInner}[1]{\langle}{\rangle}{#1}
\usepackage{xparse}
\NewDocumentCommand \inner {s m g}{
	\IfBooleanTF{#1}
	{
	\internalInner*{#2\IfValueT{#3}{,#3}}
	}
	{
	\internalInner{#2\IfValueT{#3}{,#3}}
	}
}

\DeclarePairedDelimiterXPP{\dnorm}[1]{}{\lVert}{\rVert}{_{\ast}}{#1}	
\DeclarePairedDelimiterXPP{\onenorm}[1]{}{\lVert}{\rVert}{_{1}}{#1}	
\DeclarePairedDelimiterXPP{\twonorm}[1]{}{\lVert}{\rVert}{_{2}}{#1}	
\DeclarePairedDelimiterXPP{\supnorm}[1]{}{\lVert}{\rVert}{_{\infty}}{#1}	

\newmacro{\nat}{i}	
\newmacro{\natA}{i}	
\newmacro{\natB}{j}	
\newmacro{\natC}{k}	
\newmacro{\nats}{\mathbb{N}}	
\newmacro{\N}{\nats}	

\newmacro{\integer}{a}	
\newmacro{\intA}{a}	
\newmacro{\intB}{b}	
\newmacro{\intC}{c}	
\newmacro{\integers}{\mathbb{Z}}	
\newmacro{\Z}{\integers}	

\newmacro{\rational}{r}	
\newmacro{\ratA}{r}	
\newmacro{\ratB}{s}	
\newmacro{\ratC}{t}	
\newmacro{\rationals}{\mathbb{Q}}	
\newmacro{\Q}{\mathcal Q}	

\newmacro{\real}{x}	
\newmacro{\realA}{x}	
\newmacro{\realB}{y}	
\newmacro{\realC}{z}	
\newmacro{\reals}{\mathbb{R}}	
\newmacro{\R}{\reals}	

\newmacro{\complex}{z}	
\newmacro{\complexA}{z}	
\newmacro{\complexB}{w}	
\newmacro{\complexC}{z}	
\newmacro{\complexes}{\mathbb{C}}	
\newmacro{\C}{\mathbb{C}}	

\newoplims{\argmax}{arg\,max}	
\newoplims{\argmin}{arg\,min}	
\newoplims{\intersect}{\bigcap}	
\newoplims{\union}{\bigcup}	

\newop{\aff}{aff}	
\newop{\bd}{bd}	
\newop{\bigoh}{\mathcal{O}}	
\newop{\card}{card}	
\newop{\cl}{cl}	
\newop{\conv}{conv}	
\newop{\clconv}{\overline{conv}}	
\newop{\crit}{crit}	
\newop{\diag}{diag}	
\newop{\diam}{diam}	
\newop{\dist}{dist}	
\newop{\dom}{dom}	
\newop{\gra}{\textup{gra}\:}
\newop{\zer}{\textup{zer}}
\newop{\Fix}{\textup{Fix}}
\newop{\eig}{eig}	
\newop{\ess}{ess}	
\newop{\Hess}{Hess}	
\newop{\ind}{ind}	
\newop{\im}{im}	
\newop{\intr}{int}	
\newop{\Jac}{Jac}	
\newop{\one}{\mathds{1}}	
\newop{\proj}{\textbf{Proj}}	
\newop{\prox}{prox}	
 \newop{\rank}{rank}	
\newop{\relint}{ri}	
\newop{\sign}{sgn}	
\newop{\supp}{supp}	
\newop{\Sym}{Sym}	
\newop{\tr}{tr}	
\newop{\unif}{unif}	
\newop{\vol}{vol}	

\newcommand{\id}{\mathrm{id}}
\newcommand{\Id}{\mathrm{Id}}
\newcommand{\Bm}{\textup{B}_{\mm}}	
\newcommand{\Bx}{\textup{B}_{\mm_x}}	
\newcommand{\By}{\textup{B}_{\mm_y}}	
\newcommand{\Bz}{\textup{B}_{\mm_z}}	
\newcommand{\Tm}{\textup{T}^{\mm}_{\lambda,\gamma}}

\newcommand{\Bxb}{\textup{B}_{\mm_{\bar{x}}}}		
\newcommand{\Txb}{\textup{T}^{\mm_{\bar x}}_{\lambda,\gamma}}
	
\newcommand{\Txbtt}{\textup{T}^{\mm_{x(t)}}_{\lambda(t),\gamma(t)}}	
\newcommand{\Exb}{\textup{E}^{\bar x}_{\lambda,\gamma}}
	
\newcommand{\Resolvent}[2]{J_{#1 #2}}

\newcommand{\opTx}[2]{\textup{T}^{\mm_{x}}_{#1, #2}}
\newcommand{\opT}[3]{\textup{T}^{\mm_{#1}}_{#2, #3}}

\newcommand{\T}{\textup{T}}
\newcommand{\Err}{\textup{E}}
\newcommand{\tmu}{\tilde{\mu}}

\newcommand{\Exbar}[2]{\textup{E}^{\bar x}_{#1, #2}}
\newcommand{\xbar}{\bar x}

\newop{\II}{\mathcal{I}}	
\newop{\J}{\mathcal{J}}	

\newcommand{\cf}{cf.\xspace}	
\newcommand{\eg}{e.g.,\xspace}	
\newcommand{\ie}{i.e.,\xspace}	

\newcommand{\alt}[1]{#1'}	
\newcommand{\altalt}[1]{#1''}	

\newmacro{\ball}{\mathbb{B}}	
\newmacro{\clball}{\overline{\mathbb{B}}}	
\newmacro{\sphere}{\mathbb{S}}	

\newmacro{\argdot}{\kern.5pt\boldsymbol{\cdot}\kern.5pt}	
\newmacro{\ddt}{\frac{d}{dt}}	
\newmacro{\del}{\partial}	

\newmacro{\const}{c}	
\newmacro{\constA}{a}	
\newmacro{\constB}{b}	
\newmacro{\Const}{C}	

\newmacro{\param}{\theta}	
\newmacro{\params}{\Theta}	

\newmacro{\coef}{\alpha}	
\newmacro{\coefA}{\lambda}	
\newmacro{\coefB}{\mu}	
\newmacro{\coefC}{\nu}	

\newmacro{\expA}{p}	
\newmacro{\expB}{q}	
\newmacro{\expC}{r}	

\newmacro{\precs}{\eps}		
\newmacro{\precsalt}{\delta}	

\newmacro{\asympteq}{\asymp}

\newmacro{\func}{g} 

\newmacro{\realspace}{\R^{\vdim}}	
\newmacro{\vecspace}{\mathcal{X}}	
\newmacro{\dspace}{\R^{\vdim}}	
\newmacro{\subspace}{\mathcal{W}}	

\newmacro{\coord}{i}	
\newmacro{\coordA}{i}	
\newmacro{\coordB}{j}	
\newmacro{\coordC}{k}	
\newmacro{\nCoords}{d}	
\newmacro{\dims}{\nCoords}	
\newmacro{\vdim}{\nCoords}	

\newmacro{\bvec}{e}	
\newmacro{\uvec}{u}	
\newmacro{\bvecs}{\mathcal{E}}	

\newmacro{\point}{x}	
\newmacro{\pointA}{\point}	
\newmacro{\pointB}{\point'}	
\newmacro{\pointalt}{\pointB}
\newmacro{\pointC}{\point''}	
\newmacro{\pointaltalt}{\pointC}
\newmacro{\points}{\mathcal{X}}	
\newmacro{\intpoints}{\relint\points}	

\newmacro{\base}{p}	
\newmacro{\baseA}{q}	
\newmacro{\baseB}{q}	
\newmacro{\baseC}{u}	

\newmacro{\set}{\mathcal{K}}	
\newmacro{\setA}{\set}	
\newmacro{\setB}{\alt\set}	
\newmacro{\setC}{\altalt\set}	

\newmacro{\idx}{i}
\newmacro{\idxalt}{j}
\newmacro{\idxaltalt}{l}
\newmacro{\idxaltaltalt}{k}
\newmacro{\indices}{I}
\newmacro{\indicesalt}{J}

\newmacro{\closed}{\mathcal{C}}	
\newmacro{\cpt}{\mathcal{K}}	
\newmacro{\cptalt}{\alt\cpt}	
\newmacro{\nhd}{\mathcal{U}}	
\newmacro{\nhdalt}{\W}	
\newmacro{\nhdaltalt}{\V}	
\newmacro{\nbd}{\nhd}
\newmacro{\nbdalt}{\nhdalt}
\newmacro{\nbdaltalt}{\nhdaltalt}
\newmacro{\U}{\mathcal{U}}	
\newmacro{\V}{\mathcal{V}}	
\newmacro{\W}{\mathcal{W}}	
\newmacro{\open}{\mathcal{U}}	
\newmacro{\openA}{\mathcal{U}}	
\newmacro{\openB}{\mathcal{V}}	

\newmacro{\mfld}{\mathcal{M}}	
\newmacro{\gmat}{g}	
\newmacro{\gdist}{\dist_{\gmat}}	

\newmacro{\tvec}{z}	
\newmacro{\form}{\omega}	

\newmacro{\radius}{r}
\newmacro{\Radius}{R}
\newmacro{\Radiusalt}{\widetilde{\Radius}}
\newmacro{\radiusalt}{\alt\radius}
\newmacro{\margin}{\delta}
\newmacro{\marginalt}{\alt\margin}
\newmacro{\Margin}{\Delta}
\newmacro{\connectedcomp}{\mathcal{K}}
\newmacro{\plainset}{S} 
\newmacro{\interv}{A} 
\newmacro{\domain}{D}
\newmacro{\bigcpt}{\mathcal{D}}
\newsmartmacro{\bigcptalt}{\alt\bigcpt}
\newsmartmacro{\bigcptaltalt}{\alt\alt\bigcpt}

\newmacro{\cvx}{\mathcal{C}}	
\newmacro{\subd}{\partial}	

\newop{\tspace}{T}	
\newop{\tcone}{TC}	
\newop{\dcone}{\tcone^{\ast}}	
\newop{\ncone}{NC}	
\newop{\pcone}{PC}	
\newop{\hull}{\Delta}	

\newop{\minimize}{minimize}	
\newop{\Opt}{Opt}	
\newop{\Sol}{Sol}	
\newop{\gap}{Gap}	
\newop{\orcl}{\mathsf{G}}	
\newop{\err}{\mathsf{Z}}	

\newmacro{\obj}{f}	
\newmacro{\objalt}{g}	
\newmacro{\objA}{f}	
\newmacro{\objB}{g}	
\newmacro{\sobj}{F}	

\newcommand{\sol}[1][\point]{#1^{\ast}}	

\newmacro{\gvec}{g}	
\newmacro{\gbound}{G}	

\newmacro{\oper}{A}	
\newmacro{\vecfield}{v}	
\newmacro{\vbound}{V}	

\newmacro{\lips}{L}	
\newmacro{\strong}{\alpha}	
\newmacro{\smooth}{\beta}	
\newmacro{\tmplips}{L}	
\newmacro{\tmpbound}{B}	
\newmacro{\growth}{M}	
\newmacro{\regparam}{\lambda}	

\newop{\ex}{\mathbb{E}}	
\newop{\prob}{\mathbb{P}}	
\newop{\probalt}{\mathbb{Q}}	
\newop{\var}{\mathbb{V}}	
\newop{\cov}{cov}	
\newop{\simplex}{\hull}	

\DeclarePairedDelimiterXPP{\exof}[1]{\ex}{[}{]}{}{
 #1}

\DeclarePairedDelimiterXPP{\exwrt}[2]{\ex_{#1}}{[}{]}{}{
 #2}

\DeclarePairedDelimiterXPP{\probof}[1]{\prob}{(}{)}{}{
 #1}

\DeclarePairedDelimiterXPP{\probwrt}[2]{\prob_{\!#1}}{(}{)}{}{
 #2}

\DeclarePairedDelimiterXPP{\oneof}[1]{\one}{\{}{\}}{}{#1}	

\DeclarePairedDelimiterXPP{\varof}[1]{\var}{[}{]}{}{
 #1}

\DeclarePairedDelimiterXPP{\covof}[1]{\cov}{(}{)}{}{
 #1}

\newmacro{\event}{E}	
\newmacro{\eventA}{E}	
\newmacro{\eventB}{H}	

\newmacro{\seed}{\theta}	
\newmacro{\seeds}{\Theta}	
\newmacro{\pdist}{P}	
\newmacro{\history}{\mathcal{H}}	

\newmacro{\sample}{\omega}	
\newmacro{\samples}{\Omega}	
\newmacro{\sspace}{\R^{m}}	

\newmacro{\filter}{\mathcal{F}}	
\newmacro{\probspace}{(\samples,\filter,\prob)}	

\newmacro{\mean}{\mu}	
\newmacro{\sdev}{\sigma}	
\newmacro{\variance}{\sdev^{2}}	
\newmacro{\variancealt}{s^2}

\newmacro{\covmat}{\Sigma}
\newmacro{\hessmat}{H}

\newmacro{\rv}{X}
\newmacro{\trv}{X_\Radius}
\newmacro{\gaussian}{\mathcal{N}}

\newmacro{\partition}{Z}

\newmacro{\nSamples}{n}	
\newmacro{\datapoint}{\xi}

\newmacro{\beforestart}{-1}	
\newmacro{\start}{0}	
\newmacro{\afterstart}{1}	
\newmacro{\running}{\start,\afterstart,\dotsc}	

\newmacro{\run}{n}	
\newmacro{\runA}{n}	
\newmacro{\runB}{k}	
\newmacro{\runC}{\ell}	
\newmacro{\nRuns}{N}	
\newmacro{\nRunsalt}{\alt\nRuns}	
\newmacro{\runs}{\mathcal{\nRuns}}	
\newmacro{\runalt}{\runB}	

\newmacro{\tstart}{0}	
\newmacro{\timeA}{t}	
\newmacro{\timeB}{s}	
\newmacro{\timealt}{\timeB}	
\newmacro{\timealtalt}{u}
\newmacro{\timeC}{\tau}	
\newmacro{\timeD}{\lambda}	
\newmacro{\horizon}{T}	
\newmacro{\horizonalt}{S}	

\newmacro{\seq}{a}	
\newmacro{\seqA}{a}	
\newmacro{\seqB}{b}	
\newmacro{\seqC}{c}	

\newmacro{\state}{x}	
\newsmartmacro{\accstate}{x^{\step}}	
\newmacro{\stateA}{x}	
\newmacro{\stateB}{z}	
\newmacro{\statealt}{\stateB}
\newmacro{\stateC}{y}	
\newmacro{\stateD}{p}	
\newmacro{\statealtalt}{\stateC}
\newmacro{\statealtaltalt}{\stateD}

\newmacro{\cstate}{X}
\newmacro{\cstateA}{X}
\newmacro{\cstateB}{Z}
\newmacro{\cstatealt}{\cstateB}

\newmacro{\startingpoint}{\point_\start}	

\newcommand{\curr}[1][\state]{\draft{#1_{\run}}}	

\newmacro{\mat}{M}	
\newmacro{\hmat}{H}	

\newmacro{\ones}{\mathbf{1}}	
\newmacro{\eye}{I}	
\newmacro{\identity}{\eye}	

\newmacro{\eigval}{\lambda}	

\newop{\Nash}{NE}	
\newop{\CE}{CE}	
\newop{\CCE}{CCE}	
\newop{\NI}{NI}	

\newop{\brep}{br}	
\newop{\reg}{Reg}	
\newop{\preg}{\overline{Reg}}	
\newop{\val}{val}	

\newmacro{\play}{i}	
\newmacro{\playA}{i}	
\newmacro{\playB}{j}	
\newmacro{\playC}{k}	
\newmacro{\nPlayers}{N}	
\newmacro{\players}{\mathcal{\nPlayers}}	

\newmacro{\pure}{\alpha}	
\newmacro{\pureA}{\alpha}	
\newmacro{\pureB}{\beta}	
\newmacro{\pureC}{\gamma}	
\newmacro{\nPures}{A}	
\newmacro{\pures}{\mathcal{\nPures}}	

\newmacro{\strat}{x}	
\newmacro{\stratA}{x}	
\newmacro{\stratB}{\stratA'}	
\newmacro{\stratC}{\stratA''}	
\newmacro{\strats}{\mathcal{X}}	
\newmacro{\intstrats}{\strats^{\circle}}	

\newmacro{\pay}{u}	
\newmacro{\payv}{v}	
\newmacro{\payfield}{v}	
\newmacro{\loss}{\ell}	

\newmacro{\game}{\mathcal{G}}	
\newmacro{\gameFull}{\game(\players,\points,\pay)}	

\newmacro{\fingame}{\Gamma}	
\newmacro{\fingameFull}{\Gamma(\players,\pures,\pay)}	

\newmacro{\minmax}{L}	
\newmacro{\minvar}{\point_{1}}	
\newmacro{\minvarA}{\point_{1}}	
\newmacro{\minvarB}{\minvarA'}	
\newmacro{\minvars}{\points_{1}}	
\newmacro{\maxvar}{\point_{2}}	
\newmacro{\maxvarA}{\point_{2}}	
\newmacro{\maxvarB}{\maxvarA'}	
\newmacro{\maxvars}{\points_{2}}	

\newmacro{\pot}{U}	

\newmacro{\gradientbound}{16 \dims \log 6}	

\newmacro{\hreg}{h}	
\newmacro{\breg}{D}	
\newmacro{\mprox}{P}	
\newmacro{\mirror}{Q}	
\newmacro{\fench}{F}	
\newmacro{\hstr}{K}	
\newmacro{\hrange}{H}	
\newmacro{\proxdom}{\points^{\hreg}}	

\DeclarePairedDelimiterXPP{\bregof}[2]{\breg}{(}{)}{}{#1,#2}	
\DeclarePairedDelimiterXPP{\proxof}[2]{\mprox_{#1}}{(}{)}{}{#2}	

\newmacro{\dpoint}{y}	
\newmacro{\dpointA}{y}	
\newmacro{\dpointB}{\dpointA'}	
\newmacro{\dpointC}{\dpointA''}	
\newmacro{\dpoints}{\mathcal{Y}}	

\newmacro{\dstate}{Y}	
\newmacro{\dvec}{w}	

\newmacro{\zone}{\mathbb{D}}	

\newop{\Eucl}{\Pi}	
\newop{\logit}{softmax}	
\newop{\dkl}{KL}	

\newmacro{\flowmap}{\Theta}	
\DeclarePairedDelimiterXPP{\flowof}[2]{\flowmap_{#1}}{(}{)}{}{#2}	
\DeclarePairedDelimiterXPP{\dotflowof}[2]{\dot\flowmap_{#1}}{(}{)}{}{#2}	

\newmacro{\traj}{x}	
\DeclarePairedDelimiterXPP{\trajof}[1]{\traj}{(}{)}{}{#1}	
\DeclarePairedDelimiterXPP{\difftrajof}[1]{\dot\traj}{(}{)}{}{#1}	

\newcommand{\est}[1]{\hat #1}	

\newmacro{\signal}{\est\gvec}	
\newmacro{\step}{\eta}	
\newmacro{\learn}{\eta}	
\newmacro{\tempinv}{\beta}	
\newmacro{\batch}{B}
\newmacro{\batchidx}{b}

\newmacro{\efftime}{\tau}	

\newmacro{\error}{Z}	

\newmacro{\noise}{\mathsf{U}}	
\newmacro{\snoise}{\xi}	
\newmacro{\noisepar}{\sdev}	
\newmacro{\noisevar}{\variance}	
\newmacro{\aggnoise}{\mathrm{\uppercase\expandafter{\romannumeral1}}}	
\newmacro{\supnoise}{\aggnoise_{\infty}}	
\newmacro{\maxnoise}{\aggnoise^{\ast}}	

\newmacro{\bias}{b}	
\newmacro{\drift}{b}	
\newmacro{\bbound}{B}	
\newmacro{\sbias}{\chi}	
\newmacro{\aggbias}{\mathrm{\uppercase\expandafter{\romannumeral2}}}	
\newmacro{\supbias}{\aggbias_{\infty}}	
\newmacro{\maxbias}{\aggbias^{\ast}}	

\newmacro{\sbound}{M}	
\newmacro{\aggsecond}{\mathrm{\uppercase\expandafter{\romannumeral3}}}	
\newmacro{\supsecond}{\aggsecond_{\infty}}	
\newmacro{\maxsecond}{\aggsecond^{\ast}}	

\newmacro{\mix}{\delta}	
\newmacro{\unitvec}{w}	
\newmacro{\unitvar}{W}	
\newmacro{\perturb}{z}	

\newmacro{\purequery}{\est\pure}	
\newmacro{\query}{\est\state}	
\newmacro{\pivot}{\point}	
\newmacro{\querypoint}{\est\point}	

\newmacro{\vertex}{v}	
\newmacro{\vertexA}{v}	
\newmacro{\vertexB}{w}	
\newmacro{\vertexC}{u}	
\newmacro{\nVertices}{V}	
\newmacro{\vertices}{\mathcal{V}}	

\newmacro{\edge}{e}	
\newmacro{\edgeA}{e}	
\newmacro{\edgeB}{\edgeA'}	
\newmacro{\edgeC}{\edgeA''}	
\newmacro{\nEdges}{E}	
\newmacro{\edges}{\mathcal{\nEdges}}	

\newmacro{\graph}{\mathcal{G}}	
\newmacro{\graphFull}{\graph(\vertices,\edges)}	
\newmacro{\adjmat}{A}	
\newmacro{\wmat}{W}	

\newmacro{\tree}{T}
\newmacro{\treealt}{\alt\tree}
\newmacro{\trees}{\mathcal{T}}
\newmacro{\treesalt}{\wilde\trees}

\newmacro{\mgf}{M}	
\DeclarePairedDelimiterXPP{\mgfof}[2]{\mgf_{#1}}{(}{)}{}{#2}	

\newmacro{\cgf}{K}	
\DeclarePairedDelimiterXPP{\cgfof}[2]{\cgf_{#1}}{(}{)}{}{#2}	

\newmacro{\ham}{\mathcal{H}}	
\DeclarePairedDelimiterXPP{\hamof}[2]{\ham_{#1}}{(}{)}{}{#2}	

\newmacro{\lag}{\mathcal{L}}	
\DeclarePairedDelimiterXPP{\lagof}[2]{\lag_{#1}}{(}{)}{}{#2}	

\newmacro{\mom}{p}
\newmacro{\pos}{q}
\newmacro{\vel}{v}
\newmacro{\velalt}{w}

\newmacro{\hamilt}{\mathcal{H}}
\newmacro{\hamiltalt}{\bar\hamilt}

\newmacro{\lagrangian}{\mathcal{L}}
\newmacro{\lagrangianalt}{\bar\lagrangian}

\newmacro{\curve}{\gamma}	
\DeclarePairedDelimiterXPP{\curveat}[1]{\curve}{(}{)}{}{#1}	
\DeclarePairedDelimiterXPP{\diffcurveat}[1]{\dot\curve}{(}{)}{}{#1}	

\newmacro{\curves}{\Gamma}
\DeclarePairedDelimiterXPP{\curvesat}[3]{\curves_{#1}}{(}{)}{}{#2;#3}	

\newmacro{\contcurves}{\contfuncs}
\DeclarePairedDelimiterXPP{\contcurvesat}[2]{\contcurves_{#1}}{(}{)}{}{#2}	

\newmacro{\lint}{\cstate}	
\DeclarePairedDelimiterXPP{\lintat}[1]{\lint}{(}{)}{}{#1}	

\newmacro{\qpot}{B}	
\newmacro{\qmat}{B}	
\newmacro{\energy}{E}	

\newmacro{\action}{\mathcal{S}}
\newmacro{\act}{\mathcal{S}}
\DeclarePairedDelimiterXPP{\actof}[2]{\act_{#1}}{[}{]}{}{#2}	

\newmacro{\pth}{\curve}
\newmacro{\pthalt}{\varphi}
\newmacro{\pths}{\curves}
\newmacro{\dpth}{\xi}
\newmacro{\dpthalt}{\zeta}
\newmacro{\daction}{\mathcal{A}}
\newmacro{\quasipot}{V}
\newmacro{\dquasipot}{B}
\newmacro{\symdquasipot}{C}
\newmacro{\dquasipotalt}{\widetilde{\dquasipot}}
\newmacro{\invpot}{W}
\newmacro{\dinvpot}{\energy}
\newmacro{\logmgf}{H}
\newmacro{\contfuncs}{\mathcal{C}}
\newmacro{\map}{F}
\newmacro{\rate}{\rho}
\newmacro{\level}{s}

\newmacro{\eqcl}{\mathcal{K}}
\newmacro{\neqcl}{\nComps}
\newmacro{\primvar}{\alpha}
\newmacro{\bdprimvar}{\primvar_\infty}
\newmacro{\bdvar}{\noisepar_{\infty}^{2}}
\newmacro{\exponent}{s}
\newmacro{\bdpot}{\pot_\infty}
\newmacro{\potgbound}{C}

\newmacro{\ratenhd}{\mathcal{N}}

\newmacro{\diffcurr}{\delta \curr}

\newmacro{\comp}{\mathcal{K}}
\newmacro{\iComp}{i}
\newmacro{\jComp}{j}
\newmacro{\kComp}{k}
\newmacro{\nComps}{K}

\newmacro{\iGround}{0}
\newmacro{\ground}{\comp_{\iGround}}
\newmacro{\groundstates}{\sol[\indices]}
\newmacro{\asymptstables}{AS}
\newmacro{\nontrivialattract}{NTA}

\addauthor[Pan]{PM}{MediumBlue}

\newmacro{\meas}{\mu}	
\newmacro{\altmeas}{\nu}	
\newmacro{\occmeas}{\meas}	
\newmacro{\borel}{\mathcal{B}}	
\newmacro{\size}{\delta}	
\newmacro{\toler}{\eps}	

\newcommand{\Pro}{\mathcal{P}(\Xi)}

\newcommand{\Ex}{\mathds{E}}
\newcommand{\Wass}{\mathds{W}}

\newcommand{\E}{\mathcal E}

\newcommand{\dd}{\mathrm{d}}
\newcommand{\mm}{\mathsf{m}}
\DeclareMathOperator{\m}{\mathrm{m}}

\newcommand{\X}{\mathcal{X}}

\newcommand{\cH}{{\mathcal H}}

\newcommand{\cP}{{\mathcal P}}

\newcommand{\I}{\iota}

\begin{document}

 

\title[Second order splitting dynamics]{Second order splitting dynamics for stochastic monotone inclusions with closed loop distribution}
\author
[W.~Si]
{Wutao Si$^{\ast}$}
\email{wutao.si@univ-littoral.fr}
\author
[H.~Ennaji]
{Hamza Ennaji$^{c,\ast}$}
\address{$^{c}$\,%
Corresponding author.}
\address{$^{\ast}$\,%
Univ. Grenoble Alpes, CNRS, Grenoble INP*, LJK, 38000 Grenoble, France.}
\email[Corresponding author]{hamza.ennaji@univ-grenoble-alpes.fr}
\author
[J.~Fadili]
{Jalal Fadili$^{\sharp}$}
\address{$^{\sharp}$\,%
ENSICAEN, Normandie Université, CNRS, GREYC, France.}
\email{jalal.fadili@ensicaen.fr}






\makeatletter	
\newcommand{\thmtag}[1]{	
  \let\oldthetheorem\thetheorem	
  \renewcommand{\thetheorem}{#1}	
  \g@addto@macro\endtheorem{	
    \addtocounter{theorem}{0}	
    \global\let\thetheorem\oldthetheorem}	
  }
\makeatother

\makeatletter	
\newcommand{\asmtag}[1]{	
  \let\oldtheassumption\theassumption	
  \renewcommand{\theassumption}{#1}	
  \g@addto@macro\endassumption{	
    \addtocounter{assumption}{0}	
    \global\let\theassumption\oldtheassumption}	
  }
\makeatother

\begin{abstract}
    In this paper, we investigate the problem of finding a zero of the sum of a maximal monotone operator $A$ and a cocoercive operator $\Bm$ in a Hilbert space. This formulation naturally captures stochastic optimization problems with decision-dependent distributions, often referred to as performative prediction. We propose and analyze continuous-time second-order dynamics governed by a distributionally evaluated forward-backward splitting operator. We establish the existence and uniqueness of the equilibrium point under a general uniform monotonicity assumption. In this setting, employing a vanishing viscous damping coefficient, we prove the strong convergence of the trajectories to the equilibrium, accompanied by fast asymptotic convergence rates for the velocities. Furthermore, when the regularizing operator is strongly monotone, we consider a constant Polyak-type damping coefficient and we establish global exponential convergence rates for the dynamical system.
\end{abstract}

\allowdisplaybreaks	
\acresetall	
\acused{iid}
\acused{LHS}
\acused{RHS}

\maketitle
\setcounter{tocdepth}{1}
\section{Introduction}\label{sec:Intro}
\subsection{Context and motivation}

This paper is devoted to the study, in a real Hilbert space $\cH$, of structured monotone inclusions of the form
\begin{equation}\label{eq:Intro-MI}
	\textup{find}~\xbar\in\cH:~0\in A(\Bar{x}) + \Bxb(\Bar{x}),
\end{equation}
where $A:\cH\rightrightarrows\cH$ is a maximally monotone operator, $\Bx:\cH\to\cH$ is a cocoercive operator, and $(\mm_x)_{x\in\cH}$ is a family of probability measures indexed by $x$. We assume that the solution set $\zer(A+\Bxb)$ is nonempty. A prototypical example arises when $A = \partial g$ and $B_{\m_x}(x) = \mathbb{E}_{\xi\sim \m_x}[B(x, \xi)]$ with $B(x, \xi) = \nabla f(x, \xi)$. This setting appears particularly in optimization problems with decision-dependent distributions \cite{drusvyatskiy2023}, built upon the framework of performative prediction \cite{perdomo2020,mendler2020stochastic}. A brief overview of this topic is provided in the sequel.

In the classical (deterministic) setting, one of the standard methods to tackle inclusions of the form
\begin{equation}\label{eq:Intro-MI-dit}
	\textup{find}~\xbar\in\cH:~0\in A(\Bar{x}) + B(\Bar{x}),
\end{equation}
where $A:\cH\rightrightarrows\cH$ is maximally monotone, $B:\cH\to\cH$ is cocoercive, such that $\zer(A+B)\neq\emptyset$, is to reformulate \eqref{eq:Intro-MI-dit} as a single-valued equation
\begin{equation}\label{eq:Intro-MI-T}
	\textup{find}~\xbar\in\cH:~T_{\lambda}^{A,B}(\xbar) = 0.
\end{equation}
Here, the regularized forward-backward operator $T_{\lambda}^{A,B}$ is defined as
\begin{equation}\label{eq:FB-Op}
	T_{\lambda}^{A,B} := \frac{1}{\lambda} \left[ \Id - J_{\lambda A} \circ (\Id - \lambda B) \right],
\end{equation}
where $\lambda>0$, $\Id$ denotes the identity operator, and $ J_{\lambda A}:= (\Id +\lambda A)^{-1}$ is the resolvent of $A$. In this context, it is of interest to investigate the asymptotic behavior of trajectories generated by second-order dynamics of the form
\begin{equation}\label{eq:din}
	\ddot{x}(t) + \nu(t)\dot{x}(t) + T_{\lambda}^{A,B}(x(t)) + \omega \frac{\dd }{\dd t} T_{\lambda}^{A,B}(x(t)) = 0, \quad t\geq 0,
\end{equation}
where $\nu:[0,\infty)\to [0,\infty)$ is a friction parameter and $\omega\geq 0$.

\subsection{Stochastic optimization with decision dependent distributions}
Monotone inclusions of the form \eqref{eq:Intro-MI} naturally arise when dealing with optimization problems of the type
\begin{equation}\label{eq:obj}
	\min _{x \in \cH} \mathbb{E}_{\xi \sim \m_x}[f(x, \xi)] + g(x),
\end{equation}
where $\m = (\m_x)_{x\in\cH}$ is a family of probability measures dependent on the decision variable $x$. Here, $\mathbb{E}_{\xi \sim \m_x}$ denotes the expectation with respect to $\m_x$, $f(x,\xi)$ represents the loss associated with decision $x$ and sample $\xi$, and $g$ is a proper lower semicontinuous convex function acting as a regularizer. In this setting, the goal is to learn a decision rule from a data distribution that shifts in response to the decision itself.

Problems of the form~\eqref{eq:obj} are central to \emph{performative prediction}, a framework introduced in~\cite{perdomo2020,mendler2020stochastic} and further explored in~\cite{drusvyatskiy2023,EFA,wood2023stochastic,Cutler&al,Narang&al}. The primary challenge in analyzing~\eqref{eq:obj} lies in the dependence of the measure on the decision variable, which complicates the application of standard stochastic optimization techniques. A fundamental question is therefore: what is the appropriate notion of solution for~\eqref{eq:obj}?

As established in \cite{perdomo2020,mendler2020stochastic,drusvyatskiy2023}, a suitable concept is that of \emph{equilibria}. More precisely, a point $\Bar{x}$ is called an equilibrium with respect to the distribution map $\m_{(\cdot)}$ if it satisfies
\begin{equation}\label{eq:equi_point}
	\Bar{x} \in \argmin_{x\in\cH} \mathbb{E}_{\xi\sim \m_{\Bar{x}}}[f(x, \xi)] + g(x).
\end{equation}
In other words, $\Bar{x}$ minimizes the objective for the fixed distribution $\m_{\Bar{x}}$ induced by $\Bar{x}$ itself; the performance of the decision is evaluated against the distribution it induces. Equilibria can be viewed as fixed points of the repeated minimization procedure
\begin{equation}\label{eq:rrm_intro}
	x_{t+1} \in \argmin_{x \in \cH} \Ex_{\xi\sim\mm_{x_t}}[f(x,\xi)] + g(x),
\end{equation}
a perspective that makes them tractable from an algorithmic standpoint. Consequently, rather than attempting to minimize the objective in~\eqref{eq:obj} directly, we focus on finding an equilibrium point in the sense of~\eqref{eq:equi_point}.

\subsection{Presentation of the problem} 

Throughout this paper, let $\mathcal{H}$ be a real Hilbert space endowed with the inner product $\langle \cdot, \cdot \rangle$ and the associated norm $\|\cdot\|$. Let $\Xi$ be a Polish space, \ie a separable and completely metrizable topological space, which serves as the domain of the random variable $\xi$. Our aim is to propose and analyze dynamical systems whose trajectories converge, under suitable assumptions, to a solution of the monotone inclusion:
\begin{equation}\tag{MI}\label{eq:mono_inclu}
    \textup{find}~\xbar\in\cH : 0\in A(\Bar{x}) +  \mathbb{E}_{\xi \sim \mm_{\xbar}}[B(\xbar, \xi)].
\end{equation}
To simplify the notation, given the family of measures $\mm = (\mm_x)_{x\in\cH}$, we define the single-valued operator $\Bm: \cH \to \cH$ by
\begin{equation}\label{eq:operator-Bm}
    \Bm(x) := \mathbb{E}_{\xi \sim \mm}[B(x, \xi)].
\end{equation}

To construct an appropriate dynamic for solving \eqref{eq:mono_inclu}, we adopt a variant of the forward-backward operator \eqref{eq:FB-Op} introduced in \cite{boct2024second}.

\begin{definition}[Forward-backward operator]\label{def:optT}
For any $\lambda, \gamma > 0$ and a family of measures $\mm$, we define the operator $\T_{\lambda, \gamma}^{\m}: \mathcal{H} \to \mathcal{H}$ by
\[
\T_{\lambda, \gamma}^{\m}(x) := \frac{1}{\lambda} \left[ x - J_{\gamma A} (x - \gamma \Bm(x)) \right], \quad \forall x\in\cH.
\]
\end{definition}
We observe that for any $\lambda, \gamma > 0$, the set of zeros of the operator $A + \Bm$ coincides with the set of zeros of the operator $\T_{\lambda, \gamma}^{\m}$ (\cf \cref{lem:zeros}).\\

We now consider, for $t\geq t_0>0$, two positive functions $\lambda(\cdot), \gamma(\cdot): [t_0, \infty)\to (0, \infty)$. We propose to investigate the asymptotic behavior of the trajectories generated by a second-order differential equation analogous to \eqref{eq:din}, governed by the time-dependent operator $\T_{\lambda(t), \gamma(t)}^{\m}$. More precisely, we consider the system:
\begin{equation}\tag{split-DIN}\label{eq:split_din_avd}
\ddot{x}(t)+\nu(t)\dot{x}(t)+ \Txbtt(x(t)) +\omega \frac{\dd}{\dd t}\left(\Txbtt(x(t))\right)=0.
\end{equation}
As highlighted in \cite{boct2024second}, in the acronym (split-DIN), ``split'' refers to the splitting properties induced by the forward--backward operator, whereas ``DIN'' designates the Dynamical Inertial Newton system introduced in \cite{Attouch&Laszlo2}. A mechanical interpretation underlies \eqref{eq:split_din_avd}: the dynamics may be viewed as the motion of a material point subject to a viscous damping term, ${\nu}(t)\dot{x}(t)$, combined with a correction term, $\omega \tfrac{\dd}{\dd t}\left(\Txbtt(x(t))\right)$, which plays a role analogous to Hessian-driven damping.
\begin{example}
    Consider the case where $A = N_{K} = \partial\delta_{K}$, with $K$ being a closed convex set of admissible decisions, and $B(x,\xi) = \nabla_x f(x, \xi)$ for some $f \in C^1(\cH \times \Xi)$. Then, \eqref{eq:mono_inclu} reduces to the inclusion
    \begin{equation}
    \textup{find}~\xbar\in\cH : 0\in N_{K}(\Bar{x}) +  \mathbb{E}_{\xi \sim \mm_{\xbar}}[\nabla_x f(\xbar, \xi)],
    \end{equation}
    which was studied in \cite{Cutler&al}. In this setting, the resolvent $\Resolvent{\lambda}{A}$ is simply the projection operator $\proj_{K}$ (see, \eg \cite[Examples 23.3 and 23.4]{Bauschke&Combettes}), and the operator $\T_{\lambda, \gamma}^{\m}$ reduces to
    \begin{equation}
    \label{eq:eg1_optT}
    \T_{\lambda,\gamma}^{\mm}(x)
    = \frac{1}{\lambda}\left[ x - \operatorname{proj}_{K}\left( x - \gamma\,\mathbb{E}_{\xi\sim\mm_x}[\nabla_x f(x,\xi)]\right)\right].
    \end{equation}
   This formalism can incorporate the setting of decision-dependent 
multi-player games addressed in \cite{Narang&al}. Indeed, let 
$n \geq 2$, $\cH = \prod_{i=1}^{n}\cH_i$ be a product of Hilbert 
spaces, and for each player $i \in \{1,\ldots,n\}$ let $K_i \subset \cH_i$ be a closed convex strategy set and $f_i \in C^1(\cH \times \Xi_i)$ a loss function. Each player seeks to solve
\[
    \min_{x_i \in K_i}~\mathbb{E}_{\xi_i \sim \mm^i_{x}}
    [f_i(x_i, x_{-i}, \xi_i)],
\]
where $x_{-i} = (x_1,\dotsc,x_{i-1},x_{i+1},\dotsc,x_n)$ and $\mm^i_x \in \cP(\Xi_i)$ depends on the full strategy $x\in K = \prod_{i=1}^{n} K_i$. Setting
\[
    B(x,\xi) = \bigl(\nabla_{x_1} f_1(x,\xi_1), \ldots, 
    \nabla_{x_n} f_n(x,\xi_n)\bigr), \quad 
    \xi = (\xi_1,\ldots,\xi_n),
\]
the inclusion \eqref{eq:mono_inclu} becomes the \emph{performative equilibrium} condition
\[
    \textup{find}~\xbar \in K :
    \quad
    0 \in N_{K_i}(\xbar_i) 
    + \mathbb{E}_{\xi_i \sim \mm^i_{\xbar}}
    [\nabla_{x_i} f_i(\xbar,\xi_i)],
    \quad \forall\, i = 1,\ldots,n,
\]
studied in \cite{Narang&al}. The operator $\T_{\lambda,\gamma}^{\mm}$ decomposes coordinate-wise as
\[
    \bigl(\T_{\lambda,\gamma}^{\mm}(x)\bigr)_i
    = \frac{1}{\lambda}\Bigl[x_i -
    \operatorname{proj}_{K_i}\!\Bigl(x_i
    - \gamma\,\mathbb{E}_{\xi_i \sim \mm^i_{x}}
    [\nabla_{x_i} f_i(x,\xi_i)]\Bigr)\Bigr], \quad \forall\, i = 1,\ldots,n.
\]
\end{example}
\begin{example}
Consider another common setting in machine learning where one seeks to learn a sparse predictive model. Again taking $B(x, \xi) = \nabla_x f(x, \xi)$ be the gradient of a smooth loss function $f$ evaluated on data points $\xi$, and $A = \partial g$ where $g(x) = \alpha \Vert x\Vert_1$ to promote sparsity. The inclusion \eqref{eq:mono_inclu} reads:
    \begin{equation}
        0 \in \alpha \partial \Vert \xbar\Vert_1 + \mathbb{E}_{\xi \sim \mm_{\bar{x}}}[\nabla_x f(\bar{x}, \xi)].
    \end{equation}
    In this setting, the resolvent $J_{\gamma A}$ corresponds to the celebrated soft-thresholding operator $\operatorname{prox}_{\gamma \alpha \Vert\cdot\Vert_1}$. Consequently, the operator $\T_{\lambda,\gamma}^{\mm}(x)$ takes the explicit form:
    \begin{equation}
        \T_{\lambda,\gamma}^{\mm}(x) = \frac{1}{\lambda}\left[ x - \operatorname{prox}_{\gamma \alpha \Vert\cdot\Vert_1} \left( x - \gamma\,\mathbb{E}_{\xi\sim\mm_x}[\nabla_x f(x,\xi)]\right)\right].
    \end{equation}
\end{example}
\subsection{Contributions and related works}

\paragraph{\textbf{Related works.}} The problem addressed in this paper lies at the intersection of optimization with decision-dependent distributions and continuous-time approaches to monotone inclusions via operator splitting.

\textit{Decision-dependent distributions.} The framework of decision-dependent distributions, also known as performative prediction, has gained significant attention following the seminal works of \cite{perdomo2020,mendler2020stochastic}. While early works focused mainly on the convergence of repeated risk minimization and stochastic gradient methods in convex settings \cite{drusvyatskiy2023,mendler2020stochastic}, recent extensions have broadened the scope to stochastic saddle point problems \cite{wood2023stochastic} and established asymptotic normality results for decision-dependent approximations \cite{Cutler&al}. However, the vast majority of these contributions rely on discrete-time algorithms.
The study of continuous-time dynamics in this setting was initiated in \cite{EFA}. There, the authors analyzed first and second-order dynamical systems whose trajectories converge to the equilibria of \eqref{eq:mono_inclu}. While their analysis of first-order systems covered general monotone operators, their second-order analysis incorporating inertial and Hessian-driven damping effects, was restricted to the smooth case (i.e., $A=0$ or $A$ is a smooth gradient). Specifically, they addressed equations governed by gradients of the form $\nabla G_{\mm_{x}}(x)$.

\textit{Inertial dynamics and operator splitting.} In the classical deterministic optimization setting, second-order dynamical systems have been extensively studied to accelerate convergence and dampen oscillations. The Heavy Ball method \cite{Polyak1964} and Nesterov's acceleration \cite{Nesterov1983} have inspired numerous continuous-time models \cite{Su&al, attouch2000heavy}. To attenuate the oscillations often observed in inertial systems, the Hessian-driven damping mechanism was introduced in \cite{alvarez2002second, attouch2016fast}.
However, standard Hessian damping requires the underlying potential to be twice differentiable. To handle non-smooth monotone inclusions of the form $0 \in A+B$, Attouch and L{\'a}szl{\'o} introduced in \cite{Attouch&Laszlo2} the \textit{Dynamical Inertial Newton} (DIN) system. Drawing inspiration from \cite{Attouch&Peypouquet}, they proposed to substitute the maximal monotone operator in the damping term with its Yosida regularization, thereby tackling the lack of differentiability. This methodology was later extended to a splitting framework in \cite{boct2024second,adly} via forward-backward operators. More recently, dynamics structurally similar to \eqref{eq:split_din_avd}, governed by general Lipschitz monotone operators, were investigated in \cite{BCF} in connection with Fast Krasnoselskii-Mann method, and in \cite{boct2025fast} in the context of Optimistic Gradient Descent Ascent. Furthermore, regarding asymptotic rates, \cite{boct2018convergence} established convergence results for implicit forward-backward dynamics associated with strongly monotone inclusions. Parallel to these continuous-time advances, significant progress has been made in the discrete setting with the development of accelerated inertial forward–backward algorithms (see \eg \cite{mainge2024accelerated,lorenz2015inertial,boct2024generalized,Attouch&Cabot} and the references therein). 

Our work builds upon this "split-DIN" methodology and extends these lines of research to the challenging setting of decision-dependent distributions, where the operator depends on the current state via the probability measure.\\

\paragraph{\textbf{Contributions.}}  
Our main contributions can be summarized as follows. We propose a second-order dynamic \eqref{eq:split_din_avd} that handles decision-dependent distributions while explicitly managing non-smooth constraints via the resolvent of the operator $A$. This framework generalizes the smooth setting of \cite{EFA} and extends the classical one of \cite{boct2024generalized}. 
We establish the existence and uniqueness of the equilibrium point under a general uniform monotonicity condition. Notably, our proof refines the analysis of \cite[Theorem 3.5]{EFA}. Furthermore, we prove the well-posedness of the proposed dynamic and demonstrate in \cref{thm:convergence} that, under a vanishing viscous damping profile, the trajectories converge strongly to the unique equilibrium. This result is accompanied by fast asymptotic decay rates for the velocities.
Finally, under the assumption of strong monotonicity, we consider a constant Polyak-type damping. We establish in \cref{thm:convergence_sm} that this choice yields global exponential convergence rates for both the system's trajectories and the probability measure.

\paragraph{\textbf{Organization of the paper.}} 
The paper is organized as follows. \cref{section:prelim} collects preliminary results and standard definitions. In \cref{section:existence}, we prove the existence and uniqueness of the equilibrium point under uniform and strong monotonicity and we establish well-posedness of the dynamics. \cref{section:cv-um} is dedicated to the asymptotic analysis of \eqref{eq:split_din_avd} in the general setting. More precisely, we prove the strong convergence of the trajectories to the equilibrium point and derive pointwise and integral decay rates. We strengthen these results in \cref{section:cv-sm} by performing the analysis a  strongly monotone setting in which we establish exponential convergence.     

To guide the reader, \cref{tab:summary_dynamics} provides a summary of the assumptions, the parameters, and the asymptotic convergence rates established across these two main settings. Detailed definitions of the involved quantities are provided in the subsequent sections.
\begin{table}[htpb]
    \centering
    \caption{Summary of assumptions, structural parameters, and convergence rates for the decision-dependent forward-backward dynamics across the two main settings studied in this work.}
    \label{tab:summary_dynamics}
    \renewcommand{\arraystretch}{1.6}     \resizebox{\textwidth}{!}{    \begin{tabular}{@{}lcc@{}}
        \toprule
        \textbf{Property / Parameter} & \textbf{Uniform Monotonicity} (\cref{section:cv-um}) & \textbf{Strong Monotonicity} (\cref{section:cv-sm}) \\
        \midrule
        \textbf{Regularizer $A$} & Uniformly monotone with modulus $\phi$ & Strongly monotone with modulus $\mu_A$ \\
        \textbf{Equilibrium existence} & $\phi(t) > \beta\tau t^2, \quad \forall t>0$ & $ \frac{\beta\tau}{\mu_A}<1$ \\
        \textbf{Modulus of $\Txb$} & $\Phi(s) = \frac{s}{\lambda(t)} \left( s - \psi^{-1}(s) \right)$ & $\tilde{\mu} = \frac{\gamma}{\lambda} \left( \frac{\mu_A}{1+\gamma\mu_A} \right)$ \\
        \midrule
        \textbf{Time-scaling $\lambda(t)$} &  $\lambda t^3$ with $ \lambda > \frac{4(1+\gamma\beta\tau)^2}{\alpha}$ & Constant: $\lambda > 0$ \\
        \textbf{Damping profile $\nu(t)$} & Vanishing viscous: $\frac{\alpha}{t}$ $\left(\text{with } \alpha \geq 3\right)$ & Constant à la Polyak: $2\sqrt{\tilde{\mu}}$ \\
        \textbf{Dynamical stability} & -- & $\rho = \frac{\beta\tau}{\tilde{\mu}} < \frac{\sqrt{2}\lambda}{8\gamma}$ \\
        \midrule
        \textbf{Trajectory convergence} & Strong  & Strong Exponential \\
        \textbf{Asymptotic rate on $x(t)$} & $\Psi_A(\|x(t) - \bar{x}\|) = o(1)$ & $\|x(t) - \bar{x}\| = \mathcal{O}\left( e^{-\frac{\sqrt{\tilde{\mu}}}{16} t} \right)$ \\
        \textbf{Asymptotic rate on $\mm_x(t)$} & $\Wass_1(\mm_{x(t)}, \mm_{\bar{x}}) = o(1)$ & $\Wass_1(\mm_{x(t)}, \mm_{\bar{x}}) = \mathcal{O}\left( e^{-\frac{\sqrt{\tilde{\mu}}}{16} t} \right)$ \\
        \bottomrule
    \end{tabular}%
    }
\end{table}

\section{Notation and preliminary results}\label{section:prelim}

This section establishes the mathematical framework and key definitions that will be used throughout the paper. We introduce fundamental concepts from convex analysis and operator theory that serve as building blocks for our theoretical development.

\subsection{Convex analysis}

Consider a function $g:\cH\to\R\cup\{+\infty\}$. Its domain is given by $\dom(g) = \{x\in\cH:~g(x)<\infty\}$. The class $\Gamma_{0}(\cH)$ consists of proper convex functions (bounded below with nonempty domain) that are lower semicontinuous (l.s.c) and take values in $\R\cup\{+\infty\}$. For $\alpha>0$, $g$ is called $\alpha$-strongly convex if $g - \frac{\alpha}{2}\Vert .\Vert^{2}$ is convex.

The subdifferential of $g$ is the set-valued map
\[
\partial g:x\in\cH\mapsto\{ v\in\cH:~g(y)\geq g(x) + \langle v,y-x\rangle \}.
\]
Fermat's optimality condition characterizes minimizers of $g \in \Gamma_0(\cH)$ through the condition:
\[
0\in\partial g(x^{\star}) \Leftrightarrow x^\star \in \argmin g(\cH).
\]

\begin{definition}[Differentiability]
	Let $g:\cH\to\R\cup\{\infty\}$ and $x\in\textup{int}(\dom(g))$. The function $g$ is (Fréchet) differentiable at $x$ if there exists $v\in\cH$ such that
	\[
	\lim_{h\to 0} \frac{g(x+h) - g(x) - \langle v,h\rangle}{\Vert h\Vert} = 0.
	\]
	The unique vector $v$ is the gradient, denoted $\nabla g(x)$.
\end{definition}

When $g \in \Gamma_{0}(\cH)$ is differentiable at $x$, the subdifferential reduces to $\partial g(x)=\{\nabla g(x)\}$.

\begin{definition}[$L$-smoothness]
	Let $L\geq 0$ and $g:\cH\to\R\cup\{\infty\}$. The function $g$ is $L$-smooth over $D\subset\cH$ if it is differentiable on $D$ and its gradient is Lipschitz continuous with constant $L$:
	\[
	\Vert \nabla g(x) - \nabla g(y)\Vert \leq L \Vert x - y\Vert~\text{for all}~x,y\in D.
	\]
	The class $C^{1,1}_{L}(D)$ denotes $L$-smooth functions on $D$.
\end{definition}

The proximal mapping of a function $g:\cH\to\R\cup\{+\infty\}$ is defined as
\[
\prox_{g}(x) = \argmin_{y \in \cH}\left\{g(y) + \frac{1}{2}\Vert y - x\Vert^{2}\right\}~\text{for all}~x\in\cH.
\]
When $g = \I_{K}$ (the indicator function of a closed convex set $K\subset\cH$), the proximal mapping coincides with the projection operator: $\prox_{g} = \proj_{K}$.

\subsection{Operator theory}

For a set-valued operator $A:\cH\rightrightarrows\cH$, we define: the domain $\dom(A) = \{x\in\cH:~A(x)\neq \emptyset\}$, the graph $\gra (A) = \{ [x,u]\in\cH\times\cH:~u\in Ax\}$, and the zeros set $\zer (A)  = \{x\in\cH:~0\in A(x)\}:= A^{-1}(0)$.

\begin{definition}\label{def:op_properties}
	\begin{itemize}
		\item $A$ is $\beta$-Lipschitz continuous if it is single-valued and
		\begin{equation}\label{eq:Lip_op}
			\Vert Ax - Ay \Vert \leq \beta \Vert x-y\Vert , \quad \forall x,y\in \dom A.
		\end{equation}
		
		\item $A:\cH\rightrightarrows\cH$ is monotone if
		\begin{equation}\label{eq:monotone_op}
			\langle x-y,u-v\rangle \geq 0 , \quad \forall [x,u],[y,v]\in\gra A.
		\end{equation}
		
		\item $A$ is maximal monotone if there is no monotone operator $B$ (satisfying \eqref{eq:monotone_op}) with $\gra A\subsetneq \gra B$.
		\item We say that $A$ is uniformly monotone with modulus $\phi:[0,\infty)\to [0,\infty)$ if $\phi$ is increasing, $\phi(0) = 0$, $\lim_{t\to\infty}\phi(t) = \infty$, and
		\begin{equation}\label{eq:uniformly_monotone_op}
			\langle x-y,u-v\rangle\geq  \phi\left(\Vert x-y\Vert\right) , \quad \forall [x,u],[y,v]\in\gra A.
		\end{equation}
		\item  $A$ is $\mu$-strongly monotone (with $\mu >0$) if
		\begin{equation}\label{eq:str_monotone_op}
			\langle x-y,u-v\rangle\geq \mu \Vert x-y\Vert^{2} , \quad \forall [x,u],[y,v]\in\gra A.
		\end{equation}
		
		\item A single-valued operator $B:\cH\to\cH$ is $\theta$-cocoercive (for $\theta>0$) if
		\begin{equation}\label{eq:cocoercive_op}
			\langle B(x)-B(y),x-y\rangle\geq \theta \Vert B(x)-B(y)\Vert^{2} , \quad \forall x,y\in\cH.
		\end{equation}
		
		\item For $T:D\to\cH$ with $D\subseteq\cH$, $T$ is firmly nonexpansive if for all $x,y\in D$
		\begin{equation}\label{eq:firmly-nonexpansive}
			\Vert T(x) - T(y)\Vert^2 + \Vert(\id -T)(x) - (\id - T)(y)\Vert^2 \leq \Vert x-y\Vert^2.
		\end{equation}
		A nonexpansive operator is 1-Lipschitz continuous. 
	\end{itemize}
\end{definition}

\begin{remark}
	\begin{itemize}
		\item Strong monotonicity of $A$ with constant $\mu$ is equivalent to monotonicity of $A-\mu\Id$.
		
		\item An operator $B$ is $\theta$-cocoercive if and only if its inverse $B^{-1}$ is $\theta$-strongly monotone (see \eg \cite[Example 20.31]{Bauschke&Combettes}). Additionally, when $A:\cH\rightrightarrows\cH$ is maximally monotone and $B:\cH\to\cH$ is cocoercive, the sum $A+B$ remains maximally monotone (see \eg \cite[Example 20.31]{Bauschke&Combettes}).		
		\item A direct consequence of the Cauchy-Schwarz inequality shows that a $\theta$-cocoercive operator $B$ is $\frac{1}{\theta}$-Lipschitz continuous. Indeed, for any $x,y\in\cH$:
		\[
			\Vert B(x)-B(y)\Vert\Vert x-y\Vert \geq\langle B(x)-B(y),x-y\rangle\geq \theta \Vert B(x)-B(y)\Vert^{2}.
		\]
	\end{itemize}
\end{remark}

\begin{example}
	 The subdifferential $\partial g$ of any function $g\in\Gamma_{0}(\cH)$ exemplifies a maximal monotone operator, commonly referred to as a subpotential maximal monotone operator.
\end{example}

\begin{example}
	Given a maximally monotone operator $A:\cH\rightrightarrows\cH$, its resolvent $J_A$ is single-valued and firmly nonexpansive (see \eg \cite[Propositions 23.8 and 23.10]{Bauschke&Combettes}). Conversely, an operator $T$ is firmly nonexpansive if and only if it is the resolvent of some maximally monotone operator (see \eg \cite[Corollary 23.9]{Bauschke&Combettes}).
\end{example}
For a comprehensive treatment of operator theory and convex analysis, the reader is referred to the classical monographs \cite{brezismaxmon,Bauschke&Combettes}.

\subsection{Monotone inclusions}

Consider a maximal monotone operator $A$ on $\cH$, a single-valued mapping $D:[t_0,+\infty[\times\overline{\dom(A)} \rightarrow \cH$, and the differential inclusion
\begin{equation}\label{eq:solution_MI}
\left\{\begin{array}{l}
\dot{x}(t)+A(x(t))+D(t,x(t)) \ni 0, \quad t \in [t_0, T] \\
x(t_0)=x_0\in\dom(A).
\end{array}\right.
\end{equation}

\begin{definition}\label{def:strong-solution}
A trajectory $x:[t_0,T]\to \cH$ is a strong solution of \eqref{eq:solution_MI} on $[t_0,T]$ if:
\begin{enumerate}[label=(\alph*)]
\item $x$ is continuous on $[t_0,T]$ and absolutely continuous on any compact subset of $]t_0,T[$ (hence almost everywhere differentiable);
\item $x(t)\in\dom(A)$ for almost every $t\in]t_0,T]$, and \eqref{eq:solution_MI} holds for almost every $t\in]t_0,T[$.
\end{enumerate}

A global strong solution on $[t_0,+\infty[$ is a strong solution trajectory on each compact interval $[t_0,T]$.
\end{definition}
For further details on monotone inclusions, we refer the reader to the classical monographs \cite{brezismaxmon} or \cite{Barbu}.

\subsection{Wasserstein distance}
Let $\Pro$ denote the space of probability measures on $\Xi$. For any $\mm_1, \mm_2 \in \Pro$, the $1$-Wasserstein distance is defined via the Kantorovich-Rubinstein duality as
$$
\mathbb{W}_1(\mm_1, \mm_2) = \sup _{h \in \operatorname{Lip}_1} \left| \mathbb{E}_{\xi \sim \mm_1} [h(\xi)] - \mathbb{E}_{\eta \sim \mm_2} [h(\eta)] \right|,
$$
where $\operatorname{Lip}_1$ denotes the set of $1$-Lipschitz continuous functions $h: \Xi \to \mathbb{R}$.


\subsection{Standing Assumptions}

We begin by specifying the structural properties of the operators involved in our equilibrium problem.

\begin{assumption}\label{assu:1}
The following conditions hold on the operators:
    \begin{itemize}
        \item $A: \mathcal{H} \rightrightarrows \mathcal{H}$ is a maximally monotone set-valued operator;
        \item $B: \mathcal{H} \times \Xi \rightarrow \mathcal{H}$ is single-valued, with $\xi \mapsto B(x, \xi)$ being $\m_x$-measurable for each $x \in \mathcal{H}$. There exists $\beta > 0$ such that for every $x \in \mathcal{H}$, the partial map $\xi \mapsto B(x, \xi)$ is uniformly $\beta$-Lipschitz.
    \end{itemize}
\end{assumption}

To control the behavior of the expectation operator $B_{\m_x}(x) = \mathbb{E}_{\xi \sim \m_x}[B(x, \xi)]$, we impose a regularity condition on $B$.

\begin{assumption}\label{assu:2}
There exists $\theta > 0$ such that for $\m_x$-almost every $\xi \in \Xi$, the map $x \mapsto B(x, \xi)$ is $\theta$-cocoercive. Specifically, for all $y, z \in \mathcal{H}$,
\[
    \theta \|B(y, \xi) - B(z, \xi)\|^2 \leq \langle B(y, \xi) - B(z, \xi), y - z \rangle.
\]
\end{assumption}

The pointwise cocoercivity in \cref{assu:2} provides sufficient structure to ensure that the aggregated operator inherits this property.

\begin{lemma}\label{lem:Bm-cocoercive}
Under \cref{assu:2}, the expectation operator $B_{\m_x}(x) = \mathbb{E}_{\xi \sim \m_x}[B(x, \xi)]$ is $\theta$-cocoercive.
\end{lemma}

\begin{proof}
For arbitrary $y, z \in \mathcal{H}$, we have:
\begin{align*}
\langle B_{\m_x}(y) - B_{\m_x}(z), y - z \rangle 
&= \mathbb{E}_{\xi \sim \m_x}[\langle B(y, \xi) - B(z, \xi), y - z \rangle] \\
&\overset{(i)}{\geq} \theta \mathbb{E}_{\xi \sim \m_x}[\|B(y, \xi) - B(z, \xi)\|^2] \\
&\overset{(ii)}{\geq} \theta \left\| \mathbb{E}_{\xi \sim \m_x}[B(y, \xi) - B(z, \xi)] \right\|^2 \\
&= \theta \| B_{\m_x}(y) - B_{\m_x}(z) \|^2,
\end{align*}
where $(i)$ uses \cref{assu:2}, and $(ii)$ follows from Jensen's inequality.

\end{proof}

\begin{remark}
By virtue of \cref{lem:Bm-cocoercive}, the operator $B_{\m_x}$ is Lipschitz continuous with constant $L \leq \frac{1}{\theta}$.
\end{remark}
\subsection{Reformulation via the forward-backward operator $\Tm$ }
As emphasized earlier, our main objective is to prove that the trajectories of \eqref{eq:split_din_avd} converge to the equilibria of the inclusion \eqref{eq:mono_inclu}. Our analysis relies on the fundamental observation that for any \(\lambda, \gamma > 0\), the set of zeros of the operator \(A + \Bm\) coincides exactly with the set of zeros of the operator \(\Tm\). This property is formalized in the following elementary, yet crucial, lemma.

\begin{lemma}\label{lem:zeros}
	Let $\mm \in\Pro$ and $\lambda,\gamma >0$, the zero sets coincide:
	\[
	\zer(A + \Bm) = \zer(\Tm).
	\]
\end{lemma}
\begin{proof}
	Let $x\in\zer(A + \Bm)$. Then, for $\gamma>0$, we have $-\gamma \Bm(x)\in \gamma A(x)$. Adding, $x$ in both sides, we get $x-\gamma \Bm(x)\in \gamma A(x)+x$, that is $x = \Resolvent{\gamma}{A}(x-\gamma \Bm(x))$. For $\lambda >0$, this implies that $\frac{1}{\lambda}\left(x = \Resolvent{\gamma}{A}(x-\gamma \Bm(x))\right) = \Tm(x) = 0$. Conversely, if $x\in\zer(\Tm)$, \ie $\Tm(x) = 0$, then $x = \Resolvent{\gamma}{A}(x-\gamma \Bm(x))$. By property of the resolvent (see \eg \cite[Proposition 23.2]{Bauschke&Combettes}, we obtain $x-\gamma \Bm(x)\in \gamma A(x)+x$. We get, after dividing by $\gamma>0$ that $- \Bm(x)\in  A(x)$, \ie $x\in \zer(A+\Bm)$.
\end{proof}
\begin{remark}
		Contrary to $A+\Bm$, the operator $\Tm$ is single-valued essentially, since $\Resolvent{\gamma}{A}$ is single-valued (see \eg \cite[Corollary 23.11]{Bauschke&Combettes}). This is crucial in the study of the second order differential equations.
\end{remark}

\section{Well-posedness: Existence of Equilibria and Trajectories}\label{section:existence}
In this section, we perform the analysis of the second-order dynamics \eqref{eq:split_din_avd}. More precisely, we address the existence and uniqueness of equilibria as well the existence and uniqueness of global solutions. 

First, we start with following assumption is essential for our analysis, it characterizes the sensitivity of the distribution map with respect to shifts in the decision variable. 
\begin{assumption}[Lipschitz continuity of the distribution map]\label{assu:W1}
  There exists $\tau > 0$ such that for all $x, y \in \mathcal{H}$,
  \[
  \Wass_1(\m_x, \m_y) \leq \tau\Vert x-y\Vert. 
  \]
\end{assumption}

\cref{assu:W1} is a standard regularity condition in the literature of stochastic optimization with decision-dependent distributions (see \eg \cite{drusvyatskiy2023,perdomo2020}). As emphasized in \cite{EFA}, this assumption is intimately connected to the notion of coarse Ricci curvature on metric random walk spaces.

As observed earlier, the dynamics \eqref{eq:split_din_avd} are governed by the operator $\Txbtt$, where the measure $\m_{x(t)}$ depends on the state $x(t)$ itself. Consequently, this problem cannot be treated directly within the classical framework of evolution equations governed by time-dependent monotone operators (see, \eg \cite[Proposition 6.2.1]{Haraux}). Following the strategy proposed in \cite{EFA}, we overcome this difficulty by performing a suitable reformulation. We begin by introducing the notion of an equilibrium point.

\begin{definition}[Equilibrium point]\label{def:equilibrium}
We say that $\bar{x}\in\cH$ is an equilibrium for the family of probability measures $(\mm_x)_{x \in \cH}$ if it satisfies
\begin{equation}\label{eq:equilibrium}
	0\in A(\xbar)+\Bxb(\xbar),
\end{equation}
or equivalently, if $\opT{\bar x}{\lambda}{\gamma}(\xbar) = 0$ for all $\lambda,\gamma >0$.
\end{definition}

\begin{remark}
We note that state dependent inclusions of the form $0\in A(\xbar) + B_{\xbar}(\xbar)$ were investigated in \cite{Adly&Le} using the Douglas-Rachford splitting algorithm. While conceptually related, their structural assumptions differ from ours, making their framework unsuitable for the stochastic optimization with decision-dependent distributions. In fact, contrary to \cref{assu:1,assu:2}, the authors in \cite{Adly&Le} assume that $A$ is single-valued, strongly monotone and Lipschitz continuous. Moreover, the state-dependency is embedded in the multi-valued operator $B_x$, requiring its resolvent $\Resolvent{\gamma}{B_x}$ to be Lipschitz continuous with respect to the state $x$. In contrast, our framework naturally fits with stochastic optimization with decision-dependent distributions as we embed the non-smooth constraints in the state-independent maximally monotone operator $A$ while the state-dependency is handled by the single-valued cocoercive operator $\Bx$.
\end{remark}

First, we start by gathering several useful estimates. To do so, we assume that $\xbar$ exists, we postpone the proof of the existence of equilibria to \cref{subsection:equilibria}.

\subsection{Preparatory results}\label{subsection:prep_res}
We begin by establishing some fundamental properties of the gap function $\Exb$. We first observe that, by definition, $\Exb(\bar{x}) = 0$.

\begin{lemma}\label{lemm:1}
    Under \cref{assu:1,assu:W1}, for any $y, z \in \mathcal{H}$, the following properties hold:
    \begin{enumerate}[(i)]
        \item\label{it:lemm:1-1} $\sup_{x\in\cH}\|\By(x) - \Bz(x)\| \leq \beta \tau \|y-z\|$;
        \item\label{it:lemm:1-2} $\|\Exb(x)\| \leq \frac{\gamma}{\lambda}\beta \tau \|x- \bar{x}\|$;
        \item\label{it:lemm:1-3} $\langle \Exb(x), x-\bar{x}\rangle \geq \frac{\lambda}{2}\|\Exb(x)\|^2 - \frac{(1+\gamma \beta \tau)^2}{2\lambda }\|x- \bar{x}\|^2.$
    \end{enumerate}
\end{lemma}

\begin{proof}
The property \ref{it:lemm:1-1} follows directly from \cite[Lemma~3.9 and Corollary~3.10]{EFA}. To prove \ref{it:lemm:1-2}, we see that for $x\in \cH$
    \begin{align*}
         \|\Exb(x)\| &= \|\opTx{\lambda}{\gamma}(x)-\Txb(x)\|\\
         &= \left\|\frac{1}{\lambda} \left[ x - J_{\gamma A}(x - \gamma \Bx(x)) \right] - \frac{1}{\lambda} \left[ x - J_{\gamma A}(x - \gamma \Bxb(x)) \right]\right\|\\
         &\leq \frac{\gamma}{\lambda}\|\Bx(x) - \Bxb(x)\|\\
         &\leq \frac{\gamma}{\lambda}\beta \tau \|x- \bar{x}\|,
    \end{align*}
    where the third inequality follows from 1-Lipschitz continuity of the resolvent $J_{\gamma A}$ and the last inequality follows from \ref{it:lemm:1-1}. To finish the proof; we set $z = J_{\gamma A} (u)$ with $u = x - \gamma \Bx(x)$, $\bar{z}= J_{\gamma A} (\bar{u})$ with $\bar{u} = x - \gamma \Bxb(x)$, then
    \begin{equation}\label{eq:inner}
    \begin{aligned}
    \langle \Exb(x), x-\bar{x}\rangle &= \left\langle \opTx{\lambda}{\gamma}(x)-\Txb(x), x- \bar{x}\right\rangle\\
    &= \left\langle\frac{1}{\lambda} ( x - z ) - \frac{1}{\lambda} ( x - \bar{z}), x -\bar{x}  \right\rangle\\
    & = \frac{1}{\lambda}\langle \bar{z}-z, x-\bar{x}\rangle.
    \end{aligned}
    \end{equation}
Since $J_{\gamma A}$ is firmly nonexpansive (see \eg \cite[Proposition 23.8]{Bauschke&Combettes}), then 
$\langle \bar{z}-z, \bar{u}-u\rangle \geq \|\bar{z}-z\|^2$. Therefore,
$$
\begin{aligned}
    \langle \bar{z}-z, x - \bar{x}\rangle  &= \langle \bar{z}-z, \bar{u}-u\rangle + \langle \bar{z}-z, x-\bar{x} - (\bar{u}-u)\rangle\\
    &\geq \|\bar{z}-z\|^2  - \|\bar{z}-z\| \|x-\bar{x} - (\bar{u}-u)\|\\
    &= \|\bar{z}-z\|^2  - \|\bar{z}-z\| \|x-\bar{x} - \gamma(B_{\m_x}(x) - B_{\m_{\bar{x}}}(x))\|\\
    &\overset{\ref{it:lemm:1-1}}{\geq} \|\bar{z}-z\|^2  - \|\bar{z}-z\| \Big(\|x-\bar{x}\| + \gamma\beta \tau\|x-\bar{x}\|\Big)\\
    &=  \|\bar{z}-z\|^2  - (1+\gamma \beta \tau)\|\bar{z}-z\|\|x-\bar{x}\|\\
    &\geq \|\bar{z}-z\|^2 - \frac{p}{2}\|\bar{z}-z\|^2 - \frac{(1+\gamma \beta \tau)^2}{2 p}\|x-\bar{x}\|^2,
\end{aligned}
$$
where in the last inequality we have used Young's inequality with $p>0$. Therefore, 
$$
 \langle \Exb(x), x-\bar{x}\rangle \geq \frac{1}{\lambda}(1-\frac{p}{2})\|z-\bar{z}\|^2 - \frac{(1+\gamma \beta \tau)^2}{2\lambda p}\|x- \bar{x}\|^2.
$$
Note that $\Exb(x) = \frac{1}{\lambda}(\bar{z}-z)$, we choose $p = 1$, then
$$
 \langle \Exb(x), x-\bar{x}\rangle \geq \frac{\lambda}{2}\|\Exb(x)\|^2 - \frac{(1+\gamma \beta \tau)^2}{2\lambda }\|x- \bar{x}\|^2.
$$
\end{proof}

The following lemma summarizes the key properties of the operator $\opT{\bar x}{\lambda}{\gamma}$. It is an adaptation of \cite[Lemma~2.2]{boct2024second} to our setting, where the operator is associated with the fixed equilibrium distribution $\mm_{\bar{x}}$.

\begin{lemma}\label{lemm:T}
    Suppose that \cref{assu:1,assu:2} hold and let $\xbar \in \zer(A+\Bxb)$.
    \begin{enumerate}[(i)]
        \item\label{it:lemT_1} For any $\lambda > 0$ and $\gamma \in (0, 2\theta)$, the operator $\opT{\bar x}{\lambda}{\gamma}$ is $\frac{\lambda}{2}$-cocoercive.
        
        \item\label{it:lemT_2} For all $\lambda_1, \lambda_2 > 0$, $\gamma_1, \gamma_2 \in (0, 2\theta)$, and $x, y \in \mathcal{H}$, the following estimates hold:
        \begin{equation*}
        \begin{aligned}
            &\left\|\lambda_1 \opT{\bar x}{\lambda_1}{\gamma_1}(x) - \lambda_2 \opT{\bar x}{\lambda_2}{\gamma_2}(y)\right\| \\
            &\quad \leq 4\|x-y\| + \frac{4 \theta |\gamma_1-\gamma_2|}{\gamma_1}\|\Bxb(x)\| + \frac{2|\gamma_1-\gamma_2|}{\gamma_1}\|x-\bar{x}\|,
        \end{aligned}
        \end{equation*}
        and
        \begin{equation*}
        \begin{aligned}
            &\left\| \opT{\bar x}{\lambda_1}{\gamma_1}(x) - \opT{\bar x}{\lambda_2}{\gamma_2}(y)\right\| \\
            &\quad \leq \frac{1}{\lambda_1}\left[4\|x-y\| + 4 \theta \frac{|\gamma_1-\gamma_2|}{\gamma_1}\|\Bxb(x)\| + 2 \frac{|\gamma_1-\gamma_2|}{\gamma_1}\|x-\bar{x}\|\right] \\
            &\qquad + 2 \frac{|\lambda_2-\lambda_1|}{\lambda_1 \lambda_2}\|y-\bar{x}\|.
        \end{aligned}
        \end{equation*}

        \item\label{it:lemT_3} Let $x(\cdot)$ be a strong global solution to \eqref{eq:split_din_avd}. Then, for every $t \geq t_0$, the time derivative of the operator associated with the fixed equilibrium measure satisfies:
        \[
            \left\|\frac{\dd}{\dd t}\left(\lambda(t) \opT{\bar x}{\lambda(t)}{\gamma(t)}(x(t))\right)\right\| \leq 4\|\dot{x}(t)\| + 4 \theta \frac{|\dot{\gamma}(t)|}{\gamma(t)}\|\Bxb(x(t))\| + 2 \frac{|\dot{\gamma}(t)|}{\gamma(t)}\|x(t)-\bar{x}\|.
        \]
    \end{enumerate}
\end{lemma}

\begin{remark}
    As pointed out in \cite{boct2024second}, when $B \equiv 0$ and $\lambda = \gamma$, the operator $\opT{\bar x}{\lambda}{\lambda}$ coincides with the Yosida regularization of $A$ of index $\lambda$, \ie $\opT{\bar x}{\lambda}{\lambda} = A_\lambda \coloneqq \frac{1}{\lambda} (\Id - J_{\lambda A})$. In this case, the dynamics \eqref{eq:split_din_avd} reduce to the Dynamic Inertial Newton system with Asymptotic Vanishing Damping (DIN-AVD) studied in \cite{Attouch&Laszlo2}.
\end{remark}

In order to further study the gap operator $\Exb$, we introduce the following assumption ensures the Lipschitz continuity of the operator $\Bx$.

\begin{assumption}\label{assu:Bm}
    There exists $L>0$ such that the map $x \mapsto B(x, \xi)$ is $L$-Lipschitz continuous for $\mm_x$-almost every $\xi\in \Xi$.
\end{assumption}

If \cref{assu:Bm} holds, then for any $x \in \cH$ and any $y,z\in \mathcal{H}$, we have 
\[
    \Vert\Bx(y) - \Bx(z)\Vert = \Vert\mathbb{E}_{\xi\sim\mm_x}[B(y, \xi)-B(z, \xi)]\Vert \leq L\Vert y-z\Vert.
\]
In other words, for every fixed $x\in\cH$, the operator $\Bx$ is $L$-Lipschitz continuous.
\begin{remark}
By virtue of \cref{lem:Bm-cocoercive}, \cref{assu:2} implies that 
$\Bm$ is Lipschitz continuous with constant $L \leq \frac{1}{\theta}$. 
However, \cref{assu:Bm} may hold with a sharper constant $L < \frac{1}{\theta}$, 
and is assumed independently when a tighter Lipschitz bound is required.
\end{remark}
The next lemma provides a Lipschitz-type estimate for the mapping $(x,\lambda)\mapsto \Exb(x)$, along with a bound on the temporal variation of the gap function along the trajectory.

\begin{lemma}\label{lemm:2}
    Under \cref{assu:1,assu:W1,assu:Bm}, the following properties hold:
    \begin{enumerate}[(i)]
        \item\label{it:lemm2-1} Let $\lambda_1, \lambda_2 > 0$ and $\gamma > 0$. For any $x, z \in\cH$, we have
        $$ 
        \|\Exbar{\lambda_1}{ \gamma}(x) - \Exbar{\lambda_2}{\gamma}(z)\|
        \leq \frac{|\lambda_1-\lambda_2|}{\lambda_1\lambda_2}\beta\tau\|x-\bar{x}\| + \frac{c_0}{\lambda_2}\|x-z\|,
        $$
        where $c_0 = 2+ 2\gamma L + \gamma \beta \tau$.
        
        \item\label{it:lemm2-2} Assume that $\lambda(\cdot)$ is a continuously differentiable function and that $\gamma(t) \equiv \gamma$ is constant. If $x(\cdot)$ is a global solution to \eqref{eq:sol}, then for every $t \geq t_0$, we have
        $$
        \left\|\frac{\dd}{\dd t } \Exbar{\lambda(t)}{\gamma}(x(t)) \right\| \leq \frac{ |\dot{\lambda}(t)| }{\lambda(t)^2}\beta\tau\|x(t)-\bar{x}\| + \frac{c_0}{\lambda(t)}\|\dot{x}(t)\|.
        $$
    \end{enumerate}
\end{lemma}

\begin{proof}
   (1) For any $x,z\in \mathcal{H}$, we have using the definition of $\opT{z}{\lambda}{\gamma}$
    \begin{align*}
        &\|\Exb(x) - \Exb(z)\| = \left\|\left(\opT{x}{\lambda}{\gamma}(x)-\Txb(x)\right)-\left(\opT{z}{\lambda}{\gamma}(z)-\Txb(z)\right)\right\|\\
        &= \left\|-\frac{1}{\lambda} J_{\gamma A}(x-\gamma \Bx(x)) + \frac{1}{\lambda} J_{\gamma A}(x-\gamma \Bxb(x)) + \frac{1}{\lambda}J_{\gamma A}(z - \gamma \Bz(z)) - \frac{1}{\lambda}J_{\gamma A}(z - \gamma \Bxb(z))  \right\|\\
        &\leq \frac{1}{\lambda}\|J_{\gamma A}(x - \gamma \Bxb(x)) - J_{\gamma A}(z - \gamma \Bxb(z))\| + \frac{1}{\lambda}\|J_{\gamma A}(z - \gamma \Bz(z)) - J_{\gamma A}(x - \gamma \Bx(z))\|\\
        &+ \frac{1}{\lambda}\|J_{\gamma A}(x - \gamma \Bx(z)) - J_{\gamma A}(x-\gamma \Bx(x))\|\\
        &\overset{(1)}{\leq} \frac{1}{\lambda}(1+ \gamma L)\|x-z\| + \frac{1}{\lambda}(1+\gamma \beta \tau )\|z-x\| + \frac{1}{\lambda}\gamma L \|x-z\|\\
        &= \frac{1}{\lambda}(2+ 2\gamma L + \gamma \beta \tau)\|x-z\|,
    \end{align*}
    where $(1)$ follows from the 1-Lipschitz of resolvent operator, \cref{lemm:1}, and $L$-lipschitz continuity of $\Bx$ from \cref{assu:Bm}. Furthermore, set $\lambda_1, \lambda_2>0$ and $\gamma >0$, we have 
    \begin{align*}
        &\quad \| \Exbar{\lambda_1}{\gamma}(x) - \Exbar{\lambda_2}{\gamma}(x)\| = \left\| \left(\opT{x}{\lambda_1}{\gamma}(x)-\opT{\bar{x}}{\lambda_1}{\gamma}(x)\right)-\left(\opT{x}{\lambda_2}{\gamma}(x)-\opT{\bar{x}}{\lambda_2}{\gamma}(x)\right) \right\|\\
        &= \left\|\frac{1}{\lambda_1} [ -J_{\gamma A}(x-\gamma \Bx(x)) +  J_{\gamma A}(x-\gamma \Bxb(x))] + \frac{1}{\lambda_2}[J_{\gamma A}(x - \gamma \Bx(x)) - J_{\gamma A}(x - \gamma \Bxb(x)) ] \right\|\\
        &= \left\|\frac{\lambda_1-\lambda_2}{\lambda_1\lambda_2} J_{\gamma A}(x - \gamma \Bx(x)) -  \frac{\lambda_1-\lambda_2}{\lambda_1\lambda_2} J_{\gamma A}(x - \gamma \Bxb(x))\right\| \\
        &\overset{(2)}{\leq} \frac{|\lambda_1-\lambda_2|}{\lambda_1\lambda_2}\beta \tau\|x-\bar{x}\|,
    \end{align*}
    where $(2)$ follows from the 1-Lipschitz of resolvent operator and \cref{lemm:1}. Therefore,
     \begin{align*}
        &\|\Exbar{\lambda_1}{\gamma}(x) - \Exbar{\lambda_2}{\gamma}(z)\| \leq \|\Exbar{\lambda_1}{\gamma}(x) - \Exbar{\lambda_2}{\gamma}(x)\| +\| \Exbar{\lambda_2}{\gamma}(x) - \Exbar{\lambda_2}{\gamma}(z)\|\\
        &\leq \frac{|\lambda_1-\lambda_2|}{\lambda_1\lambda_2}\beta\tau\|x-\bar{x}\| + \frac{1}{\lambda_2}(2+ 2\gamma L + \gamma \beta \tau)\|x-z\|
    \end{align*}

    To prove \ref{it:lemm2-2} we proceed as follows. Given $t, s\geq t_0$ with $t\neq s$, we get from \ref{it:lemm2-1}

    \begin{align*}
        &\frac{\|\Exbar{\lambda(t)}{\gamma(t)}(x(t)) - \Exbar{\lambda(s)}{\gamma(s)}(x(s))\|}{|t-s|}\\
        &\leq \frac{|\lambda(t)-\lambda(s)|}{|t-s|} \frac{\beta\tau}{\lambda(t)\lambda(s)}\|x(t)-\bar{x}\| + \frac{1}{\lambda(s)}(2+ 2\gamma L + \gamma \beta \tau)\frac{\|x(t)-x(s)\|}{|t-s|}
    \end{align*}
    Hence, by taking the limit as $s\to t$ we get, for any $t \geq t_0$, 
    $$
    \left\|\frac{\dd}{\dd t } \Exbar{\lambda(t)}{\gamma}(x(t)) \right\| \leq \frac{|\dot{\lambda}(t)| }{\lambda(t)^2}\beta\tau \|x(t)-\bar{x}\| + \frac{1}{\lambda(t)}(2+ 2\gamma L + \gamma \beta \tau)\|\dot{x}(t)\|.
    $$
\end{proof}
\subsection{Existence and uniqueness of equilibria}\label{subsection:equilibria}
We provide two different conditions under which the equilibrium $\bar{x}$ exists and is unique. First, we introduce the following monotonicity assumption on the operator $A$.

\begin{assumption}[Uniform monotonicity]\label{assu:A-UM}
    The operator $A$ is uniformly monotone with a modulus $\phi: [0, \infty) \to [0, \infty)$. We assume that $\phi$ is continuous, strictly increasing with $\phi(0) = 0$, that the mapping $t \mapsto \frac{\phi(t)}{t}$ is non-decreasing on $(0, \infty)$, and that it satisfies:
    \[
        \phi(t) > \beta\tau t^2, \quad \forall t>0.
    \]
\end{assumption}

A particular case of \cref{assu:A-UM}, which is a very common setting, is to require strong monotonicity on $A$. This amounts to assuming that $A$ is uniformly monotone with a quadratic modulus $\phi(t) = \mu_A t^2$.

\begin{assumption}[Strong monotonicity]\label{assu:A:SM}
    There exists $\mu_A > 0$ such that for all $x\in\cH$, $A$ is $\mu_A$-strongly monotone, and the parameters satisfy $L_S:=\frac{\beta\tau}{\mu_A}$.
\end{assumption}

A direct consequence of \cref{assu:A-UM} is the following lemma, which asserts that the operator $\opT{\bar x}{\lambda}{\gamma}$ inherits the uniform monotonicity from $A$.  

\begin{lemma}[Uniform monotonicity of the splitting operator]\label{lemma:uniform_monotony}
    Suppose that \cref{assu:1,assu:2,assu:A-UM} hold. Define the strictly increasing function $\psi(t) \coloneqq t + \gamma \frac{\phi(t)}{t}$ for $t > 0$, and $\psi(0) = 0$ (where \cref{assu:A-UM} ensures that $\psi$ is a strictly increasing bijection on $[0,+\infty)$).
    Then, for any $\lambda > 0$ and $\gamma \in (0, 2\theta)$, the operator $\opT{\bar x}{\lambda}{\gamma}$ is uniformly monotone with modulus $\Phi(t) \coloneqq \frac{t}{\lambda} \left( t - \psi^{-1}(t) \right)$. If \cref{assu:A:SM} holds instead of \cref{assu:A-UM}, then $\opT{\bar x}{\lambda}{\gamma}$ is $\tilde{\mu}$-strongly monotone with $\tilde{\mu} = \frac{\gamma}{\lambda}\frac{\mu_A}{1+\gamma\mu_A}$.
\end{lemma}

\begin{proof}
    Let $x, y \in \cH$ and define $u = (\id - \gamma \Bx)(x)$ and $v = (\id - \gamma\Bx)(y)$. Since $\gamma \in (0, 2\theta)$, we have $\Vert u - v\Vert \leq \Vert x - y\Vert$. Take $p = J_{\gamma A}(u)$ and $q = J_{\gamma A}(v)$ and assume that $p\neq q$, otherwise there is nothing to prove. By definition of the resolvent, we have that $u \in p + \gamma Ap$ and $v \in q + \gamma Aq$. Thus, using the $\phi$-uniform monotonicity of $A$, we obtain:
    \[
        \langle u - v, p - q \rangle = \Vert p - q\Vert^2 + \gamma \langle Ap - Aq, p - q \rangle \geq \Vert p - q\Vert^2 + \gamma \phi(\Vert p - q\Vert),
    \]
    which gives by the Cauchy-Schwarz inequality, $\Vert u - v\Vert \Vert p - q\Vert \geq \Vert p - q\Vert^2 + \gamma \phi(\Vert p - q\Vert)$. Dividing by $\Vert p - q\Vert > 0$, we get:
    \[
        \Vert u - v\Vert \geq \Vert p - q\Vert + \gamma \frac{\phi(\Vert p - q\Vert)}{\Vert p - q\Vert} = \psi(\Vert p - q\Vert).
    \]
    Since $\psi$ is strictly increasing and $\Vert u - v\Vert \leq \Vert x - y\Vert$, we have $\Vert p - q\Vert \leq \psi^{-1}(\Vert u - v\Vert) \leq \psi^{-1}(\Vert x - y\Vert)$. Thus, $S_\gamma = J_{\gamma A} \circ (\id - \gamma \Bx)$ satisfies $\Vert S_\gamma x - S_\gamma y\Vert \leq \psi^{-1}(\Vert x - y\Vert)$.
    
    Finally, evaluating the operator $\opT{\bar x}{\lambda}{\gamma} = \frac{1}{\lambda}(\id - S_\gamma)$:
    \begin{align*}
        \langle \opT{\bar x}{\lambda}{\gamma}(x) - \opT{\bar x}{\lambda}{\gamma}(y), x - y \rangle 
        &= \frac{1}{\lambda} \Vert x - y\Vert^2 - \frac{1}{\lambda} \langle S_\gamma x - S_\gamma y, x - y \rangle \\
        &\geq \frac{1}{\lambda} \Vert x - y\Vert^2 - \frac{1}{\lambda} \Vert S_\gamma x - S_\gamma y\Vert \Vert x - y\Vert \\
        &\geq \frac{\Vert x - y\Vert}{\lambda} \left( \Vert x - y\Vert - \psi^{-1}(\Vert x - y\Vert) \right) = \Phi(\Vert x - y\Vert).
    \end{align*}
    To finish the proof, notice that if $A$ is $\mu_A$-strongly monotone, \ie $\phi(t) = \mu_A t^2$, then $\psi(t) = (1+\gamma\mu_A)t$. Direct computations yield that $\Phi(t) = \frac{\gamma}{\lambda}\frac{\mu_A}{1+\gamma\mu_A} t^2$, and thus $\opT{\bar x}{\lambda}{\gamma}$ is $\tilde{\mu}$-strongly monotone with $\tilde{\mu} \coloneqq \frac{\gamma}{\lambda}\frac{\mu_A}{1+\gamma\mu_A}$.
\end{proof}

The main result of this section is the following.

\begin{proposition}[Existence of Equilibrium via the forward-backward operator]\label{prop:existence-equil-UM}
     Suppose that \cref{assu:1,assu:2,assu:A-UM} hold and let $\gamma \in (0, 2\theta)$. Then there exists a unique equilibrium point $\bar{x} \in \cH$ such that $\opT{\bar x}{\lambda}{\gamma}(\bar{x}) = 0$ for all $\lambda >0$, or equivalently $\bar{x} = J_{\gamma A}(\bar{x} - \gamma \Bxb(\bar{x}))$.
\end{proposition}

\begin{proof}
    First, let us define the fixed-point mapping $S: \cH \to \cH$ associated with $\opTx{\lambda}{\gamma}$, given by:
    \[
        S(x) \coloneqq J_{\gamma A}(x - \gamma \Bx(x)).
    \]
    Clearly, equilibria in the sense of \cref{def:equilibrium} are exactly the fixed points of $S$. 
    Let $x, y \in \cH$. We have:
    \[
        \Vert S(x) - S(y)\Vert = \Vert J_{\gamma A}(x - \gamma \Bx(x)) - J_{\gamma A}(y - \gamma \By(y)) \Vert.
    \]
    Recall from the proof of \cref{lemma:uniform_monotony} that $J_{\gamma A}$ satisfies the bound $\Vert J_{\gamma A}(u) - J_{\gamma A}(v)\Vert \leq \psi^{-1}(\Vert u - v\Vert)$, where $\psi(t) = t + \gamma \frac{\phi(t)}{t}$ for $t > 0$. Thus, applying this yields:
    \begin{equation}\label{eq:inv_psi}
        \Vert S(x) - S(y)\Vert \leq \psi^{-1} \Big( \Vert (x - \gamma \Bx(x)) - (y - \gamma \By(y)) \Vert \Big).
    \end{equation}
    To bound the argument of $\psi^{-1}$ in \eqref{eq:inv_psi}, we add and subtract $\gamma \Bx(y)$:
    \begin{align*}
        \Vert (x - \gamma \Bx(x)) - (y - \gamma \By(y)) \Vert 
        &= \Vert (\id - \gamma \Bx)(x) - (\id - \gamma \Bx)(y) + \gamma\left(\By(y) - \Bx(y)\right) \Vert \\
        &\leq \Vert (\id - \gamma \Bx)(x) - (\id - \gamma \Bx)(y) \Vert + \gamma \Vert \By(y) - \Bx(y) \Vert.
    \end{align*}
    Since $\Bx$ is $\theta$-cocoercive and $\gamma \in (0, 2\theta)$, the operator $\id - \gamma \Bx$ is averaged, and thus non-expansive (see, \eg \cite[Proposition 4.39]{Bauschke&Combettes}). Furthermore, using \cref{lemm:1}, we obtain:
    \[
        \Vert (\id - \gamma \Bx)(x) - (\id - \gamma \Bx)(y) \Vert \leq \Vert x - y\Vert, \quad \text{and} \quad \Vert \By(y) - \Bx(y) \Vert \leq \beta \tau \Vert x - y\Vert.
    \]    
    Plugging this into \eqref{eq:inv_psi} gives:
    \[
        \Vert S(x) - S(y)\Vert \leq \psi^{-1} \left( (1 + \gamma \beta \tau)\Vert x - y\Vert \right).
    \]
    Let us define the nonlinear mapping $\Psi: [0, \infty) \to [0, \infty)$ by $\Psi(t) \coloneqq \psi^{-1}((1 + \gamma \beta \tau)t)$. Since $\phi$ is continuous, $\Psi$ is also continuous and satisfies $\Psi(0) = 0$. The previous bound can be rewritten exactly as:
    \[
        \Vert S(x) - S(y)\Vert \leq \Psi(\Vert x - y\Vert), \quad \forall x, y \in \cH.
    \]
    To finish the proof, it remains to check the inequality $\Psi(t) < t$ for all $t > 0$. Since the function $\psi$ is strictly increasing, applying $\psi$ to both sides gives $(1 + \gamma \beta \tau)t < \psi(t)$. Substituting the definition of $\psi(t) = t + \gamma \frac{\phi(t)}{t}$, the condition $\Psi(t) < t$ is equivalent to $\phi(t) > \beta \tau t^2$. Consequently, $S$ satisfies all the requirements of the  Boyd-Wong fixed-point theorem (\cf \cref{thm:boyd-wong}) and thus admits a unique fixed point $\bar{x} \in \cH$.
\end{proof}

In the following result, we prove the existence and uniqueness of equilibria under \cref{assu:A:SM}, exploiting the structure of the sum $A+\Bm$ rather than that of the forward-backward operator $\Txb$. The proof follows the arguments developed in \cite[Theorem 3.5]{EFA}.

\begin{proposition}[Existence and uniqueness of the equilibrium]\label{prop:existence-equil-SM}
    Suppose that \cref{assu:A:SM,assu:1,assu:2} hold. Then, the single-valued mapping $S: \cH \to \cH$ defined by $S(x) = \zer(A + \Bx)$ is a strict contraction, and there exists a unique equilibrium point $\xbar \in \cH$ satisfying $0 \in A(\xbar) + \Bxb(\xbar)$.
\end{proposition}

\begin{proof}
    Let $x, y \in \cH$. Since $A$ is strongly monotone and $\Bx$ is monotone, the sum $A + \Bx$ is strongly monotone for any fixed measure. Thus, the zeros $u = S(x)$ and $v = S(y)$ are uniquely defined and we have:
    \[
        0 \in A(u) + \Bx(u) \quad \text{and} \quad 0 \in A(v) + \By(v).
    \]
    This implies there exist seclections $p \in A(u)$ and $q \in A(v)$ such that $p = -\Bx(u)$ and $q = -\By(v)$.
    Using the $\mu_A$-strong monotonicity of $A$, we have:
    \begin{equation}\label{eq:proof_eq_1}
        \mu_A \|u - v\|^2 \leq \langle p - q, u - v \rangle = \langle -\Bx(u) + \By(v), u - v \rangle.
    \end{equation}
    To handle the inner product on the right-hand side, we add and subtract $\Bx(v)$:
    \[
        \langle -\Bx(u) + \By(v), u - v \rangle = \langle -\left(\Bx(u) - \Bx(v)\right), u - v \rangle + \langle \By(v) - \Bx(v), u - v \rangle.
    \]
    Since the operator $\Bx$ is monotone, the first term is non-positive:
    \[
        \langle -\left(\Bx(u) - \Bx(v)\right), u - v \rangle \leq 0.
    \]
    Therefore, inequality \eqref{eq:proof_eq_1} becomes:
    \[
        \mu_A \Vert u - v\Vert^2 \leq \langle \By(v) - \Bx(v), u - v \rangle,
    \]
    and applying Cauchy-Schwarz inequality, we obtain:
    \[
        \mu_A \|u - v\|^2 \leq \Vert \By(v) - \Bx(v)\Vert \Vert u - v\Vert,
    \]
    that is 
    \[
        \mu_A \Vert u - v\Vert \leq \Vert \By(v) - \Bx(v)\Vert.
    \]
    Now, using \cref{lemm:1}\ref{it:lemm:1-1}, we have $\Vert \By(v) - \Bx(v)\Vert \leq \beta \tau \Vert y - x\Vert.
    $
    Combining these inequalities yields:
    \[
        \|S(x) - S(y)\| = \|u - v\| \leq \frac{\beta \tau}{\mu_A} \Vert x - y\Vert.
    \]
    Consequently, the mapping $S$ is a strict contraction with constant $L_S = \frac{\beta \tau}{\mu_A} < 1$. By the Banach Fixed-Point Theorem (\cref{thm:fixedpoint}), $S$ admits a unique fixed point $\bar{x} \in \cH$, which corresponds exactly to the unique equilibrium.
\end{proof}
\subsection{Well-posedness of \eqref{eq:split_din_avd}}\label{subsection:well-posedness}   
Now we are in a position to address well-posedness of the dynamics. As emphasized in \cite{EFA}, the dependence of the operator $\opT{\bar x}{\lambda}{\gamma}$ on the trajectory $x(t)$ makes a direct analysis of \eqref{eq:split_din_avd} delicate. Leveraging the notion of equilibrium introduced above, one can recast the dynamics as a system driven by the operator $\opT{\bar x}{\lambda}{\gamma}$ associated with the fixed equilibrium measure $\mm_{\xbar}$, subject to an additional perturbation term. 

To derive a suitable reformulation of \eqref{eq:split_din_avd}, we introduce the gap function $\Exb: \cH \to \cH$ defined by:
\begin{equation}\label{eq:gapE}
	\Exb(x) \coloneqq \opT{x}{\lambda}{\gamma}(x) - \opT{\bar x}{\lambda}{\gamma}(x).
\end{equation}
This term captures the discrepancy between the operator evaluated at the current distribution $\mm_x$ and the one evaluated at the equilibrium distribution $\mm_{\bar{x}}$. 
Substituting this decomposition into \eqref{eq:split_din_avd}, the system takes the following perturbed form:
\begin{equation}\tag{s-DIN-AVD}\label{eq:sol}
	\ddot{x}(t) + \nu(t)\dot{x}(t) + \underbrace{\opT{\bar x}{\lambda(t)}{\gamma(t)}(x(t)) + \omega \frac{\dd}{\dd t}\left(\opT{\bar x}{\lambda(t)}{\gamma(t)}(x(t))\right)}_{\text{Static Dynamics w.r.t $\mm_x$}} + \underbrace{\Exb(x(t)) + \omega \frac{\dd}{\dd t}\left( \Exb(x(t)) \right)}_{\text{Perturbation}} = 0.
\end{equation}
\begin{remark}\label{rem:gap_smooth}
	Consider the smooth case where $A=0$ and $B(x,\xi)=\nabla_x f(x,\xi)$ for some $f\in C^{1}(\cH\times\Xi)$. For any probability measure $\m$, let us define
	\[
		G_{\m}(x) \coloneqq \Ex_{\xi\sim\m}[f(x,\xi)], \qquad 
		\Bm(x) = \Ex_{\xi\sim\m}[\nabla_x f(x,\xi)] = \nabla G_{\m}(x).
	\]
	Since $J_{\gamma A}=\Id$ when $A=0$, the forward-backward operator simplifies to
	\[
		\T_{\lambda, \gamma}^{\m}(x)
		= \frac{1}{\lambda}\bigl(x - (x-\gamma \Bm(x))\bigr)
		= \frac{\gamma}{\lambda}\,\Bm(x)
		= \frac{\gamma}{\lambda}\,\nabla G_{\m}(x).
	\]
	Consequently, the gap function $\Exb(x) = \T_{\lambda, \gamma}^{\mm_{x}}(x) - \T_{\lambda, \gamma}^{\mm_{\xbar}}(x)$ can be explicitly written as
	\[
		\Exb(x) = \frac{\gamma}{\lambda}\bigl(\nabla G_{\m_x}(x) - \nabla G_{\m_{\bar{x}}}(x)\bigr).
	\]
	In particular, when $\gamma=\lambda$, the operator $\T_{\lambda, \gamma}^{\m}$ coincides with the expected gradient $\nabla G_{\m}$. In this setting, the gap $\Exb(x)$ represents exactly the difference between the gradient field under thecccurrent distribution and the gradient field under thecxequilibrium distribution.
\end{remark}

Building upon the estimates established in the previous lemmas, we are now in a position to prove the existence and uniqueness of a global solution to \eqref{eq:sol}. The proof follows standard arguments based on the Cauchy-Lipschitz theorem (see, \eg \cite[Proposition 6.2.1]{Haraux}).
\begin{theorem}\label{thm:1}
    Suppose that \cref{assu:1,assu:W1,assu:Bm} are satisfied. Let $\gamma \in (0, 2\theta)$ and assume that the parameter $\lambda: [t_0, \infty) \to (0, \infty)$ is bounded away from zero, \ie $\inf_{t \geq t_0} \lambda(t) > 0$. Furthermore, suppose that the operator $A$ satisfies either the uniform monotonicity \cref{assu:A-UM} or the strong monotonicity \cref{assu:A:SM} condition. 
    Then, for any initial conditions $(x_0, u_0) \in \cH \times \cH$, the dynamical system \eqref{eq:sol} admits a unique strong global solution $x: [t_0, \infty) \to \cH$ satisfying $x(t_0) = x_0$ and $\dot{x}(t_0) = u_0$.
\end{theorem}	
\begin{proof}
    Before detailing the proof, we assume that the viscous damping function $\nu: [t_0, \infty) \to (0, \infty)$ is continuously differentiable. This guarantees that both $\nu(t)$ and $\dot{\nu}(t)$ are continuous and bounded on any compact interval of $[t_0, \infty)$.

    \begin{enumerate}[label=\textbf{Case \arabic*.}, wide, labelindent=0pt]
        \item \textbf{$\omega > 0$.} 
        As established in \cite{Attouch&Laszlo2, boct2024second} (see also \cite[Proposition 4.3]{EFA}), for all $t \geq t_0$, the second-order system \eqref{eq:sol} is equivalent to the following first-order dynamical system:
        \begin{equation}\tag{S}\label{eq:1st-order-system}
            \left\{\begin{aligned}
                &\dot{x}(t) + \omega \Big( \opT{\bar x}{\lambda(t)}{\gamma(t)} (x(t)) + \Exbar{\lambda(t)}{\gamma(t)}(x(t)) \Big) - \left(\frac{1}{\omega}-\nu(t)\right) x(t) + \frac{1}{\omega} y(t) = 0, \\
                &\dot{y}(t) - \left(\frac{1}{\omega}-\nu(t) - \omega\dot{\nu}(t)\right) x(t) + \frac{1}{\omega} y(t) = 0.
            \end{aligned}\right.
        \end{equation}
        Consequently, solving \eqref{eq:sol} with Cauchy data $(x(t_0), \dot{x}(t_0)) = (x_0, u_0)$ is equivalent to solving the first-order problem:
        \begin{equation*}
            \left\{\begin{aligned}
                \dot{Z}(t) &= \E(t, Z(t)), \\
                Z(t_0) &= (x_0, y_0),
            \end{aligned}\right.
        \end{equation*}
        where $Z(t) = (x(t), y(t))$ and the vector field $\E$ is defined by
        \[
            \E(t, (x, y)) = 
            \begin{pmatrix}
                -\omega \left( \opT{\bar x}{\lambda(t)}{\gamma(t)} (x) + \Exbar{\lambda(t)}{\gamma(t)}(x) \right) + \left(\frac{1}{\omega}-\nu(t)\right) x - \frac{1}{\omega} y \\[8pt]
                \left(\frac{1}{\omega}-\nu(t) - \omega\dot{\nu}(t)\right) x - \frac{1}{\omega} y
            \end{pmatrix}.
        \]
        The initial condition for the auxiliary variable is given by:
        \[
            y_0 = -\omega \left( u_0 + \omega \opT{\bar x}{\lambda(t_0)}{\gamma(t_0)}(x_0) \right) + \left(1 - \omega\nu(t_0)\right)x_0 - \omega^2 \Exbar{\lambda(t_0)}{\gamma(t_0)}(x_0).
        \]

        We start by verifying the Lipschitz continuity of the vector field $\E$ with respect to the state variable $Z$. We endow the product space $\cH \times \cH$ with the standard norm $\|(x,y)\| = \|x\| + \|y\|$.
        Let $t \in [t_0, \infty)$ be fixed, and let $z = (x, y)$ and $s = (u, v)$ be two points in $\cH \times \cH$. For brevity, let us denote the time-dependent coefficients by $p(t) \coloneqq \left(\frac{1}{\omega} - \nu(t)\right)$ and $q(t) \coloneqq \left(\frac{1}{\omega} - \nu(t) - \omega\dot{\nu}(t)\right)$. 

        Using the definition of $\E$ and the triangle inequality on each component, we have:
        \begin{align*}
            \Vert\E(t,z) - \E(t,s)\Vert 
            &\leq \omega\left\Vert \opT{\bar x}{\lambda(t)}{\gamma(t)}(x) - \opT{\bar x}{\lambda(t)}{\gamma(t)}(u) \right\Vert \\
            &\quad + \omega\left\Vert \Exbar{\lambda(t)}{\gamma(t)}(x) -\Exbar{\lambda(t)}{\gamma(t)}(u) \right\Vert \\
            &\quad + \big(\vert p(t)\vert + \vert q(t)\vert\big)\Vert x-u\Vert + \frac{2}{\omega}\Vert y-v\Vert.
        \end{align*}
        By \cref{lemm:T}\ref{it:lemT_1}, the operator $\opT{\bar x}{\lambda(t)}{\gamma(t)}$ is $\frac{\lambda(t)}{2}$-cocoercive, which implies it is $\frac{2}{\lambda(t)}$-Lipschitz continuous. Similarly, it follows from \cref{lemm:2}\ref{it:lemm2-1} that the gap function satisfies:
        \[
            \left\Vert\Exbar{\lambda(t)}{\gamma(t)}(x) -\Exbar{\lambda(t)}{\gamma(t)}(u) \right\Vert \leq \frac{1}{\lambda(t)}(2+ 2\gamma L + \gamma \beta \tau)\Vert x-u\Vert.
        \]
        Let $\tilde{\lambda} \coloneqq \inf_{t\geq t_0}\lambda(t) > 0$. Injecting these estimates into the main inequality, we obtain:
        \[
            \Vert\E(t,z) - \E(t,s)\Vert \leq \left( \frac{\omega}{\tilde{\lambda}}(4+ 2\gamma L + \gamma \beta \tau) + \vert p(t)\vert + \vert q(t)\vert \right) \Vert x-u\Vert + \frac{2}{\omega}\Vert y-v\Vert.
        \]
        We define the time-dependent Lipschitz constant $K(t)$ as:
        \[
            K(t) \coloneqq \max\left\{ \frac{\omega}{\tilde{\lambda}}(4+ 2\gamma L + \gamma \beta \tau) + \vert p(t)\vert + \vert q(t)\vert, \; \frac{2}{\omega} \right\}.
        \]
        Observing that $t \mapsto p(t)$ and $t \mapsto q(t)$ are continuous on $[t_0, \infty)$, we have $K \in L_{\mathrm{loc}}^\infty([t_0,\infty), \mathbb{R})$. Thus, we conclude that:
        \[
            \Vert\E(t,z) - \E(t,s)\Vert \leq K(t)\left(\Vert x-u\Vert + \Vert y-v\Vert\right) = K(t)\Vert z-s\Vert.
        \]

        To establish the global existence of solutions, we must ensure that the vector field grows at most linearly. For any $t \in [t_0, \infty)$ and $z = (x,y)$, evaluating $\E$ at $(x,y)$ and comparing it to the equilibrium point $\bar{x}$ yields:
        \[
            \Vert\E(t,z)\Vert \leq \omega \left\Vert \opT{\bar x}{\lambda(t)}{\gamma(t)} (x) + \Exbar{\lambda(t)}{\gamma(t)}(x) \right\Vert + (\vert p(t)\vert + \vert q(t)\vert)\Vert x\Vert + \frac{2}{\omega}\Vert y\Vert.
        \]
        Recall that $\opT{\bar x}{\lambda(t)}{\gamma(t)}(\bar{x}) = 0$ by definition of the equilibrium, and $\Exbar{\lambda(t)}{\gamma(t)}(\bar{x}) = 0$. Using the Lipschitz bounds established previously, we estimate the operator term:
        \begin{align*}
            \left\Vert \opT{\bar x}{\lambda(t)}{\gamma(t)} (x) + \Exbar{\lambda(t)}{\gamma(t)}(x) \right\Vert
            &\leq \left\Vert \opT{\bar x}{\lambda(t)}{\gamma(t)}(x) - \opT{\bar x}{\lambda(t)}{\gamma(t)}(\bar{x})\right\Vert + \left\Vert\Exbar{\lambda(t)}{\gamma(t)}(x) - \Exbar{\lambda(t)}{\gamma(t)}(\bar{x})\right\Vert \\
            &\leq \left( \frac{2}{\lambda(t)} + \frac{\gamma\beta\tau}{\lambda(t)} \right) \Vert x-\bar{x}\Vert\\
            &\leq \frac{2+\gamma \beta\tau}{\tilde{\lambda}} (\Vert x\Vert + \Vert\bar{x}\Vert).
        \end{align*}
        Injecting this back into the estimate for $\|\E(t,z)\|$ and defining:
        \[
            P(t) \coloneqq \max\left\{ \frac{\omega(2+\gamma \beta\tau)}{\tilde{\lambda}} + \vert p(t)\vert + \vert q(t)\vert, \frac{2}{\omega},\frac{\omega(2+\gamma \beta\tau)}{\tilde{\lambda}}\Vert\bar{x}\Vert \right\},
        \]
        we obtain the sub-linear growth condition:
        \[
            \Vert\E(t,z)\Vert \leq P(t)(1+\Vert z \Vert).
        \]
        Since $p$ and $q$ are continuous bounded functions on bounded intervals, $P \in L_{\mathrm{loc}}^1([t_0,\infty), \mathbb{R})$. Consequently, the conditions of the non-autonomous Cauchy-Lipschitz theorem (see \eg \cite[Proposition 6.2.1]{Haraux}) are satisfied. We deduce the existence of a unique strong global solution $t \mapsto x(t)$ defined on $[t_0, \infty)$. More precisely, the pair $(x, y)$ is locally absolutely continuous on $[t_0, \infty)$, meaning that $x \in W^{2,1}_{\mathrm{loc}}([t_0, \infty); \cH)$. Thus, the second-order dynamics \eqref{eq:sol} holds for almost every $t \geq t_0$.

        \item \textbf{$\omega = 0$.} 
        By setting $y(t) = \dot{x}(t)$, \eqref{eq:sol} is straightforwardly equivalent to the first-order system:
        \begin{equation*}
        \left\{\begin{aligned}
            \dot{Z}(t) &= \E(t, Z(t)), \\
            Z(t_0) &= (x_0, y_0),
        \end{aligned}\right.
        \end{equation*}
        where $Z(t) = (x(t), y(t))$ and the vector field is defined for any $(x,y) \in \cH \times \cH$ by:
        \[
            \E(t,(x, y)) = 
            \begin{pmatrix}
                y \\
                -\nu(t) y - \left( \opT{\bar x}{\lambda(t)}{\gamma(t)} (x) + \Exbar{\lambda(t)}{\gamma(t)}(x) \right)
            \end{pmatrix}.
        \]
        Reasoning exactly as in the previous case, we easily verify that $\E$ satisfies the conditions of the non-autonomous Cauchy-Lipschitz theorem (Lipschitz continuity in space, local integrability in time). This guarantees the global existence and uniqueness of the solution.
    \end{enumerate}
    This concludes the proof.
\end{proof}
\section{Asymptotic Convergence Analysis under Uniform Monotonicity}\label{section:cv-um}
In this section, we analyze the long-term behavior of the trajectories. Throughout, we assume that $\omega \geq 0$ and $\nu(t) = \alpha/t$ with $\alpha > 1$. Recall that, based on the reformulation established earlier, the evolution system \eqref{eq:sol} is given by:
\begin{equation}\tag{s-DIN-AVD}\label{eq:sole}
    \ddot{x}(t) + \frac{\alpha}{t} \dot{x}(t) + \opT{\bar x}{\lambda(t)}{\gamma(t)}(x(t)) + \omega \frac{\dd}{\dd t}\left(\opT{\bar x}{\lambda(t)}{\gamma(t)}(x(t))\right)
    +\Exb(x(t)) + \omega \frac{\dd}{\dd t}\left( \Exb(x(t)) \right) = 0,
\end{equation}
for $t > t_0 > 0$.

The main convergence properties of the trajectories generated by \eqref{eq:sole} as $t \to \infty$ are summarized in \cref{thm:convergence}. Before stating the main result, we introduce the following auxiliary lemma which will be instrumental in the energy analysis.

\begin{lemma}\label{lemm:3}
Let $h(t)$ be a differentiable function defined for $t \geq t_0 > 0$ such that:
\begin{enumerate}[(i)]
    \item\label{lemma3_it:1} $ \left\| 2t^2 h(t) + t^3 h'(t) \right\| = o\left(\frac1t\right),\ \mathrm{as}\ t\to \infty$;
    \item\label{lemma3_it:2} $\left\| h(t) \right\| = \mathcal{O}\left( \frac{1}{t^3} \right)$;
    \item\label{lemma3_it:3} $\left\| h'(t) \right\| = \mathcal{O}\left( \frac{1}{t^4} \right)$.
\end{enumerate}
Then it follows that:
\[
\left\| h(t) \right\| = o\left( \frac{1}{t^3} \right).
\]
\end{lemma}

\begin{proof}
Define, for $t>t_0$,
\[
r(t) := 2t^2 h(t) + t^3 h'(t).
\]
By assumption \ref{lemma3_it:1}, we have $\vert r(t)\vert = o\left( \frac{1}{t} \right).$
Dividing both sides of the definition of \( r(t) \) by \( t^3\neq 0 \), we obtain the first-order linear differential equation:
\begin{equation}\label{eq:lemma_ode1}
h'(t) + \frac{2}{t} h(t) = \frac{r(t)}{t^3}.
\end{equation}
We solve it via an integrating factor. Multiplying \eqref{eq:lemma_ode1} by $t^2$, we get
\[
t^2 h'(t) + 2t h(t) = \frac{r(t)}{t},
\]
which is equivalent to:
\[
\frac{d}{dt} \left( t^2 h(t) \right) = \frac{r(t)}{t}.
\]
Integrating from \( t_0 \) to \( t \), we have:
\[
t^2 h(t) = \int_{t_0}^t \frac{r(s)}{s} ds + C,
\]
for some constant \( C \). Therefore,
\[
h(t) = \frac{1}{t^2} \int_{t_0}^t \frac{r(s)}{s} ds + \frac{C}{t^2}.
\]
Now we analyze the asymptotic behavior of \( h(t) \). Since \( \|r(s)\| = o(1/s) \), it follows that:
\[
\left\| \frac{r(s)}{s} \right\| = o\left( \frac{1}{s^2} \right),
\]
and thus $\int_{t_0}^{\infty} \frac{r(s)}{s} \dd s := I < \infty$. Hence, $\int_{t_0}^{t} \frac{r(s)}{s} \dd s = I - \int_{t}^{\infty} \frac{r(s)}{s} \dd s$. Since $\Vert r(s)/s\Vert = o(1/s^2)$ we have 
\[
\Big\Vert \int_{t}^{\infty} \frac{r(s)}{s} \dd s\Big\Vert \leq \int_{t}^{\infty}\Big\Vert \frac{r(s)}{s} \Big\Vert \dd s = o(1/t).
\]
Consequently $h(t) = \frac{1}{t^2} (I - o(1/t) + C) = C/t^2 + o(1/t^3)$. Now, suppose that $C \neq 0$. Then
\[
\left\| h(t) \right\| \geq \left\| \frac{C}{t^2} \right\| - o\left( \frac{1}{t^3} \right) = \frac{\|C\|}{t^2} - o\left( \frac{1}{t^3} \right).
\]
Thus,
\[
t^3 \|h(t)\| \geq \|C\| t - o(1) \to \infty \quad \text{as } t \to \infty,
\]
which contradicts \( \|h(t)\| = \mathcal{O}(1/t^3) \). Hence \( C = 0 \), and we obtain:
\[
h(t) = \frac{1}{t^2} \int_{t_0}^t \frac{r(s)}{s} ds = o\left( \frac{1}{t^3} \right).
\]
This completes the proof.
\end{proof}

The main result of this section is the following theorem.

\begin{theorem}\label{thm:convergence}
    Suppose that \cref{assu:1,assu:2,assu:W1,assu:Bm,assu:A-UM} hold, and let $\bar{x}$ denote the unique equilibrium point. Consider the second-order dynamics \eqref{eq:sole} with parameters satisfying the following conditions:
    \begin{equation}\label{eq:parameters-cv}
        \alpha \geq 3, \quad \omega \geq 0, \quad \gamma(t) \equiv \gamma \in (0, 2\theta), \quad \text{and} \quad \lambda(t) = \lambda t^3 \quad \text{with} \quad \lambda > \frac{4(1+ \gamma \beta \tau)^2}{\alpha}.
    \end{equation}
    Then, for any trajectory $x: [t_0,\infty)\to \mathcal{H}$ solution to \eqref{eq:sole}, the following properties hold:
    \begin{enumerate}[(i)]
        \item\label{it:mr-convergence} \textit{(Strong convergence and implicit rates).} The trajectory $x(t)$ is bounded and converges strongly as $t\to \infty$, to the unique equilibrium point $\bar{x}$. Moreover, defining the map $\Psi_A(s) \coloneqq s(s-\psi^{-1}(s))$ inherited from the uniform monotonicity, we have:
        \[
            \Psi_A(\Vert x(t) - \bar{x}\Vert) = o\left( 1 \right).
        \]
        
        \item\label{it:mr-integral} \textit{(Integral estimates).} The velocity and acceleration satisfy:
        \[
            \int_{t_0}^{\infty} t\|\dot{x}(t)\|^2 \dd t < \infty \quad \text{and} \quad \int_{t_0}^{\infty} t^3 \|\ddot{x}(t)\|^2 \dd t < \infty.
        \]
        
        \item\label{it:mr-ptw-est} \textit{(Pointwise estimates).} As $t\to \infty$, we have:
        \[
            \|\dot{x}(t)\| = o\left(\frac1t\right), \quad \|\ddot{x}(t)\| = \mathcal{O}\left(\frac{1}{t^2}\right).
        \]
        Furthermore, the operator and the gap function satisfy:
        \[
            \|\opT{\xbar}{\lambda(t)}{\gamma}(x(t))\| = o\left(\frac{1}{t^3}\right), \quad \|\Exbar{\lambda(t)}{\gamma}(x(t))\| = \mathcal{O}\left(\frac{1}{t^3}\right),
        \]
        and their time derivatives satisfy:
        \[
            \left\|\frac{\dd}{\dd t}\opT{\xbar}{\lambda(t)}{\gamma}(x(t))\right\| = \mathcal{O}\left(\frac{1}{t^4}\right), \quad \left\|\frac{\dd}{\dd t} \Exbar{\lambda(t)}{\gamma}(x(t))\right\| = \mathcal{O}\left(\frac{1}{t^4}\right).
        \]
    \end{enumerate}
\end{theorem}
\begin{remark}
	Before proceeding with the proof, let us justify the differentiability of the Lyapunov energies below. In light of \cref{thm:1}, the solution satisfies $x \in W^{2,1}_{\mathrm{loc}}([t_0, \infty); \cH)$, so that  $x(\cdot)$ and $\dot{x}(\cdot)$ are absolutely continuous on bounded intervals. Furthermore, \cref{lemm:T}\ref{it:lemT_1} ensures that $\Txb$ is cocoercive, and hence Lipschitz continuous. Consequently, the composition mapping $t\mapsto\Txb(x(t))$ is absolutely continuous and thus differentiable almost everywhere and so are all the Lyapunov energies considered in the proofs of \cref{thm:convergence,thm:convergence_sm}. 
\end{remark}
\begin{proof}
    \textbf{Lyapunov analysis.} To simplify the presentation, we set 
    $\T := \opT{\bar x}{\lambda(t)}{\gamma(t)} (x(t))$ and $\Err := \Exbar{\lambda(t)}{\gamma(t)}$, and when convenient, we use the notation $\partial_t$ instead of $\frac{\dd }{\dd t}$. Then \eqref{eq:sole} is rewritten as
    \begin{equation}\label{eq:sol1}
    \ddot{x}(t)+\frac{\alpha}{t} \dot{x}(t)+ \T(x(t)) +\omega \frac{\dd}{\dd t}\left(\T(x(t))\right) + \Err(x(t)) + \omega \frac{\dd}{\dd t}\left( \Err(x(t)) \right) =0.
\end{equation}
    Define, for $t\in [t_0,\infty)$ the following Lyapunov function
    \begin{equation}\label{eq:E}
    \mathcal{E}(t):=\frac12\|b(t)(x(t)-\bar{x}) + t(\dot{x}(t) + \omega \T(x(t)) + \omega \Err(x(t)) )\|^2 + \frac{\delta(t)}{2}\|x(t)-\bar{x}\|^2,
    \end{equation}
    where $b(t)$ and $\delta(t)$ are two positive functions that will be suitably specified to ensure that $\E$ is nonincreasing.
    Differentiation of $\mathcal{E}$ with respect to time yields
\begin{equation*}
    \begin{aligned}
        \dot{\mathcal{E}}(t) = \bigg\langle\ 
            & b(t)(x(t) - \bar{x}) + t\left[\dot{x}(t) + \omega \T(x(t)) + \omega \Err(x(t))\right], 
             \dot{b}(t)(x(t) - \bar{x}) + b(t)\dot{x}(t)\\
             &+\dot{x}(t) + \omega \T(x(t)) + \omega \Err(x(t)) + t\left[\ddot{x}(t) + \omega \partial_t\T(x(t)) + \omega \partial_t\Err(x(t))\right] 
        \bigg\rangle\\
        &+ \frac12 \dot{\delta}(t)\|x(t)-\bar{x}\|^2 + \delta(t)\langle \dot{x}(t), x(t)-\bar{x}\rangle.
    \end{aligned}
\end{equation*}
After reduction and employing~\eqref{eq:sol1}, we obtain
\begin{equation*}
    \begin{aligned}
       \dot{\mathcal{E}}(t) &= \left[b(t)\dot{b}(t) + \frac12 \dot{\delta}(t)\right]\|x(t)-\bar{x}\|^2  + t\left[b(t)+1-\alpha \right]\|\dot{x}(t)\|^2\\
       & + \left[(b(t) +1 - \alpha)b(t) + \delta(t)+t\dot{b}(t)\right]\langle x(t)-\bar{x}, \dot{x}(t)\rangle\\
       &+ 2\omega t (\omega-t)\langle \T(x(t)), \Err(x(t))\rangle + (\omega-t)\omega t\|\T(x(t))\|^2 + (\omega-t)\omega t\|\Err(x(t))\|^2 \\ & +\left[\omega t \dot{b}(t) + (\omega -t)b(t) \right]\langle x(t)-\bar{x}, \T(x(t))\rangle 
       + \left[(b(t) + 1-\alpha)\omega t + (\omega -t)t \right]\langle \dot{x}(t), \T(x(t))\rangle \\
       &+\left[\omega t \dot{b}(t) + (\omega -t)b(t) \right]\langle x(t)-\bar{x}, \Err(x(t))\rangle 
       + \left[(b(t) + 1-\alpha)\omega t + (\omega -t)t \right]\langle \dot{x}(t), \Err(x(t))\rangle .
    \end{aligned}
\end{equation*}
Here we choose 
\begin{equation}\label{eq:b-detla}
	b(t) = 1+\frac{1}{t},\ \delta(t) = -(b(t) + 1-\alpha)b(t) -t\dot{b}(t) = \alpha-2 +\frac{\alpha-2}{t} - \frac{1}{t^2}.
\end{equation}
 We then have
\begin{equation*}
    \begin{aligned}
        \dot{\mathcal{E}}(t) &= -\frac{\alpha}{2t^2}\|x(t)-\bar{x}\|^2 + \left[1-(\alpha-2)t\right]\|\dot{x}(t)\|^2  + (\omega-t)\omega t\|\T(x(t)) + \Err(x(t))\|^2\\
        &  +(\omega-1-t)\left\langle x(t)-\bar{x}, \T(x(t))\right\rangle +  \left[ -t^2 - (\alpha-3)\omega t + \omega \right]\langle \dot{x}(t), \T(x(t))\rangle  \\
       & +(\omega-1-t)\left\langle x(t)-\bar{x}, \Err(x(t))\right\rangle +  \left[ -t^2 - (\alpha-3)\omega t + \omega \right]\langle \dot{x}(t), \Err(x(t))\rangle .
    \end{aligned}
\end{equation*}
Now, using  $\frac{\lambda(t)}{2}$-cocoercivity of $\T$, which follows from \cref{lemm:T}-\ref{it:lemT_1}, and the fact that $T(\bar{x}) = 0$, we get, for all $t\geq t_1 = \max\{t_0, \omega\}$,
\begin{align*}
(\omega-1-t)\langle x(t)-\bar{x}, \T(x(t))\rangle &\leq(\omega-1-t) \frac{\lambda(t)}{2}\|\T(x(t))\|^2.
\end{align*}
It follows from \cref{lemm:1}-\ref{it:lemm:1-3} that
$$
(\omega-1-t)\langle x(t)-\bar{x}, \Err(x(t))\rangle \leq (\omega-1-t)\frac{\lambda(t)}{2}\|\Err(x)\|^2 - (\omega-1-t)\frac{(1+\gamma \beta \tau)^2}{2\lambda(t) }\|x(t)- \bar{x}\|^2.
$$
Therefore,
\begin{equation}\label{eq:dE}
    \begin{aligned}
        \dot{\mathcal{E}}(t) &\leq \left[-\frac{\alpha}{2t^2}  - (\omega - 1 - t)\frac{(1+\gamma \beta \tau)^2}{2\lambda(t)} \right]\|x(t)-\bar{x}\|^2 + \left[1-(\alpha-2)t\right]\|\dot{x}(t)\|^2  \\
        &  +(\omega-1-t) \frac{\lambda(t)}{2}\|\T(x(t))\|^2 +  \left[ -t^2 - (\alpha-3)\omega t + \omega \right]\langle \dot{x}(t), \T(x(t))\rangle  \\
       & +(\omega-1-t) \frac{\lambda(t)}{2}\|\Err(x(t))\|^2 +  \left[ -t^2 - (\alpha-3)\omega t + \omega \right]\langle \dot{x}(t), \Err(x(t))\rangle .
    \end{aligned}
\end{equation}
Since $\alpha \geq 3$, we may find $p_1>0$ such that $ \alpha > 2+p_1$. This being said, there exists $t_2 > 0$ such that for all $t \geq t_2$, we have \[
0< t+1-\omega \leq 2t,~\mbox{and}~1-(\alpha-2)t \leq -p_1 t.
\]
Since $\lambda(t) = \lambda t^3$ with $\lambda > \frac{4(1+ \gamma \beta \tau)^2}{\alpha}$, we deduce that, for all $t \geq\max\{t_1, t_2\}$,
\[
-\frac{\alpha}{2t^2}  - (\omega - 1 - t)\frac{(1+\gamma \beta \tau)^2}{2\lambda(t)} < 0,
\]
and \eqref{eq:dE} becomes
\begin{equation*}
    \begin{aligned}
        \dot{\mathcal{E}}(t) \leq -p_1 t \|\dot{x}(t)\|^2 &+  \frac{(\omega-1-t)\lambda(t)}{2}
 \|\T(x(t))\|^2 + \left[ -t^2 - (\alpha-3)\omega t + \omega \right]\left\langle \dot{x}(t), \T(x(t))\right\rangle\\
 &+  \frac{(\omega-1-t)\lambda(t)}{2}
 \|\Err(x(t))\|^2 + \left[ -t^2 - (\alpha-3)\omega t + \omega \right]\left\langle \dot{x}(t), \Err(x(t))\right\rangle.
    \end{aligned}
\end{equation*}
Choose $\epsilon$ such that $0< \epsilon < \min\{p_1, 1/2\}$, we obtain
\begin{equation}
    \begin{aligned}
        &\dot{\mathcal{E}}(t) + \frac{\epsilon}{2}t\|\dot{x}(t)\|^2 + \frac{\epsilon}{2}t\lambda(t)\|\T(x(t))\|^2+\frac{\epsilon}{2}t\lambda(t)\|\Err(x(t))\|^2\\
        &\leq \left[-p_1 +\frac{\epsilon}{2}\right]  t \|\dot{x}(t)\|^2 +   \frac{\omega -t -1+\epsilon t}{2}\lambda(t)
 \|\T(x(t))\|^2 +  \left[ -t^2 - (\alpha-3)\omega t + \omega \right]\left\langle \dot{x}(t), \T(x(t))\right\rangle\\
 &+   \frac{\omega -t -1+\epsilon t}{2}\lambda(t)
 \|\Err(x(t))\|^2 +  \left[ -t^2 - (\alpha-3)\omega t + \omega \right]\left\langle \dot{x}(t), \Err(x(t))\right\rangle.
    \end{aligned}
\end{equation}
By the choice of $\epsilon$, we have $-p_1+\epsilon/2 \leq -p_1/2$, and $\frac{\omega -t -1+\epsilon t}{2} \leq -\frac14 t + \frac{\omega-1}{2}$, then for all $t\geq \max\{t_1, t_2\}$, we have
\begin{equation}\label{eq:de5}
    \begin{aligned}
        &\dot{\mathcal{E}}(t) + \frac{\epsilon}{2}t\|\dot{x}(t)\|^2 + \frac{\epsilon}{2}t\lambda(t)\|\T(x(t))\|^2 + \frac{\epsilon}{2}t\lambda(t)\|\Err(x(t))\|^2\\
        &\leq  -\frac{p_1}{2}  t \|\dot{x}(t)\|^2 +\left[-\frac14 t\lambda(t) + \frac{\omega-1}{2}\lambda(t)\right]
\|\T(x(t))\|^2+ \left[ -t^2 - (\alpha-3)\omega t + \omega\right]\left\langle \dot{x}(t), \T(x(t))\right\rangle\\
&+ \left[-\frac14 t\lambda(t) + \frac{\omega-1}{2}\lambda(t)\right]
\|\Err(x(t))\|^2+ \left[ -t^2 - (\alpha-3)\omega t + \omega\right]\left\langle \dot{x}(t), \Err(x(t))\right\rangle
    \end{aligned}.
\end{equation}
Set 
$$
R(t): = \left[ -t^2 - (\alpha-3)\omega t + \omega \right]^2 - 4\frac{p_1}{4}t\left[\frac14 t\lambda(t) - \frac{\omega-1}{2}\lambda(t)\right].
$$
Since $\lambda(t) = \lambda t^3$, then $R(t) = -\frac{\lambda p_1 }{4}t^5 + \mathcal{O}(t^4)$. Therefore, 
there exists $t_3>0$ such that for all $t\geq t_3$, $R(t)< 0$. Consequently, the right hand side of~\eqref{eq:de5} is nonpositive. For all $t \geq \max\{t_1, t_2,t_3\}$,
we have
\begin{equation}\label{eq:int}
\dot{\mathcal{E}}(t) + \frac{\epsilon}{2}t\|\dot{x}(t)\|^2 + \frac{\epsilon}{2}t\lambda(t)\|\T(x(t))\|^2 + \frac{\epsilon}{2}t\lambda(t)\|\Err(x(t))\|^2\leq 0.
\end{equation}
\textbf{Estimates.} 
First recall the choice of $\delta(t)$ in the Lyapunov function \eqref{eq:E}. Indeed, from \eqref{eq:b-detla}
\[
\delta(t) = -(b(t) + 1-\alpha)b(t) -t\dot{b}(t) = \alpha-2 +\frac{\alpha-2}{t} - \frac{1}{t^2},
\]
taking into account the fact that $\alpha > 2 + p_1$ for $p_1>0$, we deduce that there exists $t_4>0$ such that for all $t\geq t_4$, we have $\delta(t) \geq p_1$.  Let $T := \max\{t_1, t_2,t_3, t_4\}$, then for all $t\geq T$, we obtain by integrating \eqref{eq:int} on the interval $[T, t]$
\begin{equation}
\mathcal{E}(t)+\frac{\epsilon}{2} \int_{T}^t s\|\dot{x}(s)\|^2 \dd s+\frac{\epsilon}{2} \int_{T}^t s\lambda(s)\left\| \T^{\m_{\bar{x}}}_{\lambda(s), \gamma}(x(s))\right\|^2 \dd s  +\frac{\epsilon}{2} \int_{T}^t s\lambda(s)\left\|\Err^{{\bar{x}}}_{\lambda(s), \gamma}(x(s))\right\|^2 \dd s \leq \mathcal{E}\left(T\right).
\end{equation}
From this we immediately deduce that
\begin{align}
    &\sup_{t\ge t_0}\Vert x(t) - \bar{x}\Vert<\infty,\\
    & \sup_{t\geq t_0} \left\|(1+\frac{1}{t})(x(t)-\bar{x}) + t\left(\dot{x}(t) + \omega \opT{\bar x}{\lambda(t)}{\gamma(t)} (x(t)) + \omega\Exbar{\lambda(t)}{\gamma(t)}(x(t))\right)\right\|^2 < \infty,\label{eq:14}\\
    & \int_{t_0}^{\infty} t\|\dot{x}(t)\|^2\dd t < \infty,\label{eq:11}\\
    & \int_{t_0}^{\infty} t^4\left\|\T^{\m_{\bar{x}}}_{\lambda(t), \gamma}(x(t))\right\|^2 \dd t< \infty, \ \int_{t_0}^{\infty} t^4\left\|\Err^{{\bar{x}}}_{\lambda(t), \gamma}(x(t))\right\|^2 \dd t< \infty\label{eq:13}.
    \end{align}
Moreover, it follows from \cref{lemm:T} that $\T:=\opT{\bar x}{\lambda(t)}{\gamma(t)}$ is $\lambda(t)/2$-cocoercive, which implies also $2/\lambda(t)$-Lipschitz continuity, and therefore
$$
\|\T(x(t))\| = \|\T(x(t)) - \T(\xbar)\| \leq \frac{2}{\lambda(t)} \|x(t)-\bar{x}\|.
$$
Taking into account the boundedness of $\|x(t)-\bar{x}\|$ and $\lambda(t) = \lambda t^3$, we deduce that
\begin{equation}\label{esti:T}
\|\T(x(t))\| = \mathcal{O}\left(\frac{1}{t^3}\right),\ \mathrm{as}\ t\to \infty.
\end{equation}
It follows from \cref{lemm:1,lemm:2} that 
\begin{equation}\label{esti:E}
\begin{aligned}
&\|\Err(x(t))\| \leq \frac{\gamma}{\lambda(t)}\beta \tau\|x(t) - \bar{x}\| =  \mathcal{O}\left(\frac{1}{t^3}\right),\ \mathrm{as}\ t\to \infty\\
& \left\|\partial_t\Err(x(t)) \right\| \leq  \frac{ |\dot{\lambda}(t)| }{\lambda(t)^2}\beta\tau\|x(t)-\bar{x}\| + \frac{c_0}{\lambda(t)}\|\dot{x}(t)\| =  \mathcal{O}\left(\frac{1}{t^4}\right),\ \mathrm{as}\ t\to \infty
\end{aligned}
\end{equation}
Exploiting \eqref{eq:14}, \eqref{esti:T}, \eqref{esti:E} and the boundedness of $\|x(t)-\bar{x}\|$, it follows that
\begin{equation}
    \|\dot{x}(t)\| = \mathcal{O}\left(\frac1t\right),\ \mathrm{as}\ t\to \infty.
\end{equation}
On one hand, since $\gamma>0$ is constant, we deduce from \cref{lemm:T}-\ref{it:lemT_3} that
\begin{equation}\label{eq:esti_dT}
\begin{aligned}
\left\|\frac{\dd}{\dd t}\left(\lambda(t) \T(x(t))\right)\right\| &\leq 4\|\dot{x}(t)\| = \mathcal{O}\left(\frac1t\right),\ \mathrm{as}\ t\to \infty.
\end{aligned}
\end{equation}
On the other hand, we have
$$
\left\|\frac{\dd}{\dd t} \left( \lambda(t) \T(x(t)) \right)\right\|=\left\|\dot{\lambda}(t) \T(x(t))+\lambda(t) \frac{\dd}{\dd t} \T(x(t))\right\|,
$$
which yields in light of \eqref{esti:T} and \eqref{eq:esti_dT}
$$
\left\|\lambda(t) \partial_t\T(x(t))\right\|
 \leq \left\|\partial_t\left( \lambda(t) \T(x(t)) \right)\right\| + \left\|\dot{\lambda}(t) \T(x(t))\right\| = \mathcal{O}\left(\frac{1}{t}\right),\ \mathrm{as}\ t\to \infty,
$$
and thus
\begin{equation}\label{eq:order}
\left\|\frac{\dd}{\dd t} \T(x(t))\right\| = \mathcal{O}\left(\frac{1}{t^4}\right),\ \mathrm{as}\ t\to \infty.
\end{equation}
By combining equation \eqref{eq:sol1} and \cref{lemm:1,lemm:2} we obtain, as $t\to \infty$,
\begin{equation}\label{eq:ddx}
\begin{aligned}
   &\| \ddot{x}(t)\| = \left\|\frac{\alpha}{t} \dot{x}(t)+ \T(x(t)) +\omega \frac{\dd}{\dd t}\T(x(t))+ \Err(x(t)) + \omega \frac{\dd}{\dd t}\Err(x(t))\right\|\\
  & \leq \frac{\alpha}{t}\|\dot{x}(t)\| + \|\T(x(t))\| + \omega\left\| \frac{\dd}{\dd t}\left(\T(x(t))\right) \right\| + \|\Err(x(t))\| + \omega \left\|\frac{\dd}{\dd t}\Err(x(t))\right\|\\
  &= \mathcal{O}\left(\frac{1}{t^2}\right).
   \end{aligned}
\end{equation}
Multiplying by $t^3$ in the first equality in \eqref{eq:ddx} and using Young's inequality, we get
$$
\begin{aligned}
   t^3\| \ddot{x}(t)\|^2 &\leq 5\alpha^2 t\|\dot{x}(t)\|^2 + 5t^3 \|\T(x(t))\|^2 + 5\omega^2t^3\left\| \partial_t\left(\T(x(t))\right) \right\|^2\\
   &+ 5t^3\|\Err(x(t))\|^2 + 5\omega^2 t^3 \left\|\partial_t\Err(x(t))\right\|^2.
   \end{aligned}
    $$
    According to~\eqref{eq:11},~\eqref{esti:T},~\eqref{esti:E} and~\eqref{eq:order}, we obtain that 
    \begin{equation}\label{eq:int2}
    \int_{t_0}^{\infty} t^3 \|\ddot{x}(t)\|^2d t < \infty.
    \end{equation}
    To finish the proof of \ref{it:mr-integral} and \ref{it:mr-ptw-est}, we note that 
    $$
    \begin{aligned}
    \frac{\dd}{\dd t}t^2 \|\dot{x}(t)\|^2 &= 2t \|\dot{x}(t)\|^2 + 2t^2 \langle \ddot{x}(t), \dot{x}(t)\rangle\\
    &\leq 2t \|\dot{x}(t)\|^2 + t^3 \|\ddot{x}(t)\|^2 + t\|\dot{x}(t)\|^2 = t^3 \|\ddot{x}(t)\|^2 + 3t\|\dot{x}(t)\|^2.
    \end{aligned}
    $$
It follows from~\eqref{eq:11} and~\eqref{eq:int2} that $t^3 \|\ddot{x}(t)\|^2 + 3t\|\dot{x}(t)\|^2 \in L^1([t_0,\infty], \mathbb{R})$. Therefore, the limit $\lim_{t\to \infty} t^2 \|\dot{x}(t)\|^2$ exists. Using~\eqref{eq:11} again, we have
$$
\int_{t_0}^{\infty} \frac1t (t^2 \|\dot{x}(t)\|^2)\dd t = \int_{t_0}^{\infty} t \|\dot{x}(t)\|^2 \dd t < \infty,
$$
from which we deduce that $\lim_{t\to \infty} t^2 \|\dot{x}(t)\|^2 = 0$ and consequently
$$
\|\dot{x}(t)\| = o\left( \frac1t\right)\ \mathrm{as}\ t\to \infty.
$$
It follows from~\eqref{eq:esti_dT} that
\begin{equation}\label{eq:dlambda_T}
\left\|\partial_t\left(\lambda(t) \T(x(t))\right)\right\| = o\left(\frac1t\right).
\end{equation}
Again, since
$$
\left\|\partial_t \left( \lambda(t) \T(x(t)) \right)\right\|=\left\|\dot{\lambda}(t) \T(x(t))+\lambda(t) \partial_t\T(x(t))\right\|,
$$
we set $g(t) = \T(x(t))$. We have from \eqref{eq:dlambda_T} that 
$ \|2t^2 g(t) + t^3  \frac{d}{d t} g(t)\| = o(1/t)$ where $\|g(t)\| = \mathcal{O}(1/t^3)$, $ \frac{\dd}{\dd t} g(t) = \mathcal{O}(1/t^4)$ following from \eqref{esti:T} and \eqref{eq:order}. Then applying \cref{lemm:3} to $g$, we get
\begin{equation}\label{est11}
    \|g(t)\| = \left\| \T(x(t))\right\| = o\left(\frac{1}{t^3}\right)\ \mathrm{as}\ t\to \infty,
\end{equation}
as desired.

    \textbf{Strong convergence of trajectories.} 
    Contrary to the standard monotone setting relying on Opial's lemma to establish weak convergence, our framework allows us to directly deduce the strong convergence of the trajectories by exploiting the uniform monotonicity of the operator. 
    
    Indeed, recall that from \cref{lemma:uniform_monotony}, the operator $\T \coloneqq \opT{\bar x}{\lambda(t)}{\gamma}$ satisfies the uniform monotonicity property:
    \[
        \langle \T(x(t)) - \T(\bar{x}), x(t) - \bar{x} \rangle \geq \frac{1}{\lambda(t)} \Psi_{A}(\Vert x(t) - \bar{x}\Vert),
    \]
    where $\Psi_{A}(s) \coloneqq s \left( s - \psi^{-1}(s) \right)$ is a strictly increasing function vanishing only at $0$.
    Since $\T(\bar{x}) = 0$, applying Cauchy-Schwarz inequality gives:
    \[
        \frac{1}{\lambda(t)} \Psi_{A}(\Vert x(t) - \bar{x}\Vert ) \leq \Vert\T(x(t))\Vert \Vert x(t) - \bar{x}\Vert,
    \]
  that is 
    \begin{equation}\label{eq:Psi}
        \Psi_{A}(\Vert x(t) - \bar{x}\Vert) \leq \lambda t^3 \Vert\T(x(t))\Vert \Vert x(t) - \bar{x}\Vert.
    \end{equation}
    Since the trajectory $x(t)$ is bounded, $\sup_{t \geq t_0} \Vert x(t) - \bar{x}\Vert < \infty$. Moreover, we have thanks to \eqref{est11} that $\Vert \T(x(t))\Vert = o(1/t^3)$, which implies that $t^3 \Vert \T(x(t))\Vert \to 0$ as $t \to \infty$. 
    We deduce that:
    \[
        \Psi_{A}(\Vert x(t) - \bar{x}\Vert) = o(1) \quad \text{as } t \to \infty.
    \]
    Since $\Psi_{A}$ is continuous, strictly increasing, and vanishes only at $0$, we get that $\Vert x(t) - \bar{x}\Vert \to 0$ as $t \to \infty$. This gives the strong convergence of the trajectory to the unique equilibrium $\bar{x}$.
\end{proof}
An immediate consequence of \cref{thm:convergence} concerns the asymptotic behavior of the family $(\mm_{x(t)})_t$. The following corollary establishes the convergence of the measure associated to the trajectory in the $\Wass_1$ distance.
\begin{corollary}\label{cor:measure_cv1}
    Under the assumptions of \cref{thm:convergence}, the decision-dependent probability measure $\mm_{x(t)}$ converges to the equilibrium measure $\mm_{\bar{x}}$ in the $\Wass_1$-distance:    
    \[
        \Wass_1(\mm_{x(t)}, \mm_{\bar{x}})  = o(1),~\text{as}~t \to \infty.
    \]
\end{corollary}
\begin{proof}
    We know from \cref{assu:W1} that:
    \begin{equation}\label{eq:cor_w1}
        \Wass_1(\mm_{x(t)}, \mm_{\bar{x}}) \leq \tau \Vert x(t) - \bar{x}\Vert,
    \end{equation}
and by \eqref{eq:Psi} 
        \begin{equation}\label{eq:Psi_cor}
        \Psi_A(\Vert x(t) - \bar{x}\Vert) \leq \lambda t^3 \Vert\T(x(t))\Vert \Vert x(t) - \bar{x}\Vert.
    \end{equation}
   The pointwise estimates in \cref{thm:convergence}\ref{it:mr-ptw-est} ensures that $\Vert\T(x(t))\Vert = o(1/t^3)$ as $t \to \infty$. Consequently, the product $\lambda t^3 \Vert\T(x(t))\Vert$ is $o(1)$. Since the trajectory is bounded, the entire right-hand side in \eqref{eq:Psi_cor} is $o(1)$, and thus:
    \[
        \Psi_A(\Vert x(t) - \bar{x}\Vert) = o(1) \quad \text{as } t \to \infty.
    \]
   Applying $\Psi_A^{-1}$ to this asymptotic bound gives $\Vert x(t) - \bar{x}\Vert = \Psi_A^{-1}(o(1)) = o(1).$ We conclude that $\Wass_1(\mm_{x(t)}, \mm_{\bar{x}}) \to 0$ as $t \to \infty$.
\end{proof}

\section{Exponential Convergence under Strong Monotonicity}\label{section:cv-sm}

In the previous section, we established the strong convergence of the trajectories under the general framework of uniform monotonicity of \cref{assu:A-UM}. The convergence rate was implicitly governed by the modulus $\Psi_A$. In this section, we focus on the highly common setting where the operator $A$ is strongly monotone, \ie under \cref{assu:A:SM}. As shown previously in \cref{lemma:uniform_monotony}, under \cref{assu:A:SM}, the splitting operator $\Txb$ becomes strongly monotone with modulus:
\[
    \tilde{\mu} \coloneqq \frac{\gamma}{\lambda} \frac{\mu_A}{1+\gamma\mu_A} > 0.
\]
It is well known that in the strongly monotone setting, the asymptotic vanishing damping $\nu(t) = \alpha/t$ is no longer the optimal choice. Instead, drawing inspiration from Polyak's Heavy Ball method for strongly convex optimization, we can take $\nu(t)$ to be a constant calibrated to the modulus of strong monotonicity. Specifically, we consider the critical damping parameter $\nu(t) \equiv 2\sqrt{\tilde{\mu}}$ (see, \eg \cite{Attouch-Chbani-Fadili-Riahi,Attouch-Fadili-Kungertsev}), and we fix $\lambda$ and $\gamma$ to be constant positive real numbers. The decision-dependent dynamical system \eqref{eq:sol} becomes:
\begin{equation}\tag{s-DIN-SM}\label{eq:dynamic-sm}
    \ddot{x}(t) + 2\sqrt{\tmu} \dot{x}(t)
    + \Txb(x(t)) + \omega \frac{\dd}{\dd t}\bigl(\Txb(x(t))\bigr)
    + \Exb(x(t)) + \omega \frac{\dd}{\dd t}\bigl( \Exb(x(t)) \bigr) = 0,
\end{equation}
for $t > t_0 \geq 0$. Note that the existence and uniqueness of a strong global solution to \eqref{eq:dynamic-sm} follow directly from \cref{thm:1}.

To perform the Lyapunov analysis, we introduce the energy function $\mathcal{V}: [t_0, \infty) \to \mathbb{R}^+$ defined by:
\begin{equation}\label{eq:Lyapunov_SM}
    \mathcal{V}(t) = \langle \Txb(x(t)), x(t)-\xbar \rangle + \frac{1}{2}\|v(t)\|^2,
\end{equation}
where the auxiliary velocity variable is defined as
\[
    v(t) \coloneqq \sqrt{\tmu}(x(t) - \bar{x}) + \dot{x}(t).
\]

We are now ready to state the main result of this section.

\begin{theorem}\label{thm:convergence_sm}
    Suppose that \cref{assu:1,assu:2,assu:W1,assu:Bm,assu:A:SM} hold. Let $x: [t_0, \infty) \to \cH$ be a trajectory solution of \eqref{eq:dynamic-sm}. Assume that $\rho \coloneqq \frac{\beta\tau}{\tmu}$ and the damping coefficient $\omega$ satisfy the following conditions:
    \begin{subequations}\label{eq:conditions_globales}
        \begin{align}
            0 \leq \omega &< \min\left(\frac{1}{2\sqrt{\tmu}}, \frac{\lambda\sqrt{\tmu}}{2\sqrt{2} c_{0}}\right), \label{eq:cond_omega} \\
            \rho &< \frac{\sqrt{2}\lambda}{8\gamma}, \label{eq:cond_rho}
        \end{align}
    \end{subequations}
    where $c_0 = 2+ 2\gamma L + \gamma \beta \tau$.

    Then, the following properties hold:
    \begin{enumerate}[label=(\roman*)]
        \item\label{it:hd-it1} \textit{(Exponential decay).} For all $t \geq t_0$, the energy satisfies:
        \[
            \mathcal{V}(t) \leq \mathcal{V}(t_0) e^{-\frac{\sqrt{\tmu}}{8}(t-t_0)}.
        \]
        Consequently, we obtain the following bounds on the trajectory and the operator:
        \[
            \tmu\Vert x(t)-\bar{x}\Vert^2 \leq \langle \Txb(x(t)), x(t)-\xbar\rangle \leq \mathcal{V}(t_0) e^{-\frac{\sqrt{\tmu}}{8}(t-t_0)},
        \]
        and
        \[
            \frac{\lambda}{2}\Vert\Txb(x(t))\Vert^2 \leq \langle \Txb(x(t)), x(t)-\xbar\rangle \leq \mathcal{V}(t_0) e^{-\frac{\sqrt{\tmu}}{8}(t-t_0)}.
        \]
        Moreover, there exists a constant $C_1 > 0$ such that the velocity decays as:
        \[
            \Vert\dot{x}(t)\Vert \leq C_1 e^{-\frac{\sqrt{\tmu}}{16}t}.
        \]

        \item\label{it:hd-it2} \textit{(Integral estimate).} There exists a constant $C_2 > 0$ such that for all $t \geq t_0$:
        \[
            e^{-\sqrt{\tmu} t}\int_{t_0}^{t} e^{\sqrt{\tmu}s}\|\Txb(x(s))\|^2 \dd s \leq C_2 e^{-\frac{\sqrt{\tmu}}{8}t}.
        \]
    \end{enumerate}
\end{theorem}

Let us make several remarks before proving this result.
\begin{remark}
    \begin{itemize}
    \item The energy function $\mathcal{V}$ differs from the one typically used in standard inertial dynamics (e.g., \cite[Theorem 3.5]{EFA}) in two main aspects: the absence of $\Txb$ in the definition of the velocity variable $v(t)$, a simplification enabled by the coercivity property established in \cref{lemm:T}\ref{it:lemT_1}, and the presence of the term $\langle\Txb(x(t)), x(t)-\xbar\rangle$. 
    
    In particular, in the optimization setting where $A=0$ and $B(x,\xi)=\nabla_x f(x,\xi)$ (with $f \in C^{1}(\cH\times\Xi)$ and $G_{\mm}(x) \coloneqq \Ex_{\xi\sim\mm}[f(x,\xi)]$ is convex), setting $\lambda = \gamma$ yields $\opT{\bar{x}}{\lambda}{\gamma}(x) = \nabla G_{\mm_{\bar{x}}}(x)$. Consequently, by convexity:
    \[
        G_{\mm_{\bar{x}}}(x(t)) - \min_{\cH}G_{\mm_{\bar{x}}} \leq \langle \Txb(x(t)), x(t)-\xbar \rangle.
    \]
    Thus, our result in \cref{thm:convergence_sm}\ref{it:hd-it1} directly recovers the exponential convergence rate of the objective function values.

    \item In contrast to works addressing inertial dynamics in the presence of errors and perturbations, such as \cite{Attouch-Fadili-Kungertsev,Attouch&Laszlo2}, \cref{thm:convergence_sm} requires no integrability assumptions on the error terms $\Exb$ and $\frac{\dd }{\dd t}\Exb$. Indeed, thanks to \cref{lemm:1,lemm:2}, these error terms are structurally controlled via Lipschitz continuity and are essentially absorbed into the dissipation terms of the Lyapunov analysis.
    \end{itemize}
\end{remark}
		
\begin{proof}
    To lighten the presentation, we denote $\T \coloneqq \Txb$ and $\Err \coloneqq \Exb$, and we use the notation $\partial_t$ instead of $\frac{\dd }{\dd t}$.
    Differentiating $\mathcal{V}(t)$ with respect to time yields:
    \begin{equation}
        \begin{aligned}
            \dot{\V}(t) &= \langle \partial_t \T (x(t)),x(t) - \xbar\rangle + \langle \T(x(t)),\dot{x}(t)\rangle +\langle v(t),\sqrt{\tmu}\dot{x}(t) + \ddot{x}(t) \rangle\\
            & =\langle \partial_t \T (x(t)),x(t) - \xbar\rangle + \langle \T(x(t)),\dot{x}(t)\rangle \\
            &\quad +\langle v(t),-\sqrt{\tmu}\dot{x}(t) - \T(x(t)) -\omega \partial_t \T(x(t)) -\Err(x(t)) -\omega\partial_t \Err(x(t)) \rangle.
        \end{aligned}
    \end{equation}
    Replacing $v(t)$ by its expression $\sqrt{\tmu}(x(t) - \bar{x}) + \dot{x}(t)$ and rearranging the terms, we arrive at:
    \begin{align}
        &\dot{\V}(t) +\sqrt{\tmu}\left( \sqrt{\tmu}\langle\dot{x}(t),x(t)-\bar{x}\rangle +\Vert\dot{x}(t)\Vert^2 +\langle \T(x(t)),x(t)-\bar{x}\rangle \right) \nonumber\\
        &= - \langle v(t),\Err(x(t))+\omega\partial_t\Err(x(t))\rangle +\delta\langle \partial_t \T (x(t)),x(t) - \xbar\rangle - \omega\langle\dot{x}(t),\partial_t\T(x(t))\rangle, \label{eq:der_V}
    \end{align} 
    where $\delta \coloneqq 1-\omega\sqrt{\tmu}$. Notice that $\delta >0$ under the assumption $\omega < \frac{1}{2\sqrt{\tmu}}$ from \eqref{eq:cond_omega}. By the strong monotonicity of $\T$, we have:
    \[
        -\omega\left\langle\dot{x}(t), \partial_t \T(x(t))\right\rangle = -\lim_{h \rightarrow 0}\omega\left\langle\frac{x(t+h)-x(t)}{h}, \frac{\T(x(t+h))-\T(x(t))}{h}\right\rangle \leq 0.
    \] 
    Thus, \eqref{eq:der_V} becomes:
    \begin{equation}\label{eq:V1}
        \dot{\V}(t) + \sqrt{\tmu} \Theta(t) \leq \Vert v(t)\Vert\Vert\Err(x(t))+\omega\partial_t\Err(x(t))\Vert + \delta\Vert \partial_t \T (x(t))\Vert \Vert x(t) - \xbar\Vert,
    \end{equation}
    where we defined:
    \begin{align*}
        \Theta(t) \coloneqq \langle \T(x(t)),x(t)-\bar{x}\rangle + \sqrt{\tmu} \langle\dot{x}(t),x(t)-\bar{x}\rangle +\Vert\dot{x}(t)\Vert^2.
    \end{align*}
    We immediately rewrite $\Theta(t)$ in terms of $\V(t)$ as:
    \[
        \Theta(t) = \V(t) +\frac{1}{2}\Vert\dot{x}(t)\Vert^2 -\frac{\tmu}{2}\Vert x(t) -\xbar\Vert^2.
    \]
    Consequently, \eqref{eq:V1} becomes:
    \begin{align}
        &\dot{\V}(t) + \sqrt{\tmu} \V(t) + \frac{\sqrt{\tmu}}{2}\Vert\dot{x}(t)\Vert^2 -\frac{ \tmu\sqrt{\tmu}}{2}\Vert x(t) -\xbar\Vert^2 \nonumber \\
        &\leq \Vert v(t)\Vert\Vert\Err(x(t))+\omega\partial_t\Err(x(t))\Vert + \delta\Vert \partial_t \T (x(t))\Vert \Vert x(t) - \xbar\Vert. \label{eq:V1_bis}
    \end{align}
    Using the $\tmu$-strong monotonicity of $\T$ and discarding the quadratic term in the definition of $\V(t)$, we can lower bound the energy by:
    \[
        \V(t) = \frac{1}{4}\V(t) + \frac{3}{4}\V(t) \geq \frac{1}{4}\V(t) + \frac{3}{4}\langle \T(x(t)), x(t) - \bar{x}\rangle \geq \frac{1}{4}\V(t) + \frac{3\tmu}{4}\Vert x(t) - \bar{x}\Vert^2.
    \]
    Injecting this into the $\sqrt{\tmu}\V(t)$ term on the left-hand side of \eqref{eq:V1_bis}, we end up with:
    \begin{align}
        &\dot{\V}(t)+\frac{\sqrt{\tmu}}{4}\V(t) + \frac{\sqrt{\tmu}}{2}\Vert\dot{x}(t)\Vert^2 + \frac{\tmu\sqrt{\tmu}}{4}\Vert x(t) - \bar{x}\Vert^2 \nonumber\\
        &\leq \Vert v(t)\Vert\Vert\Err(x(t))+\omega\partial_t\Err(x(t))\Vert + \delta\Vert \partial_t \T (x(t))\Vert \Vert x(t) - \xbar\Vert. \label{eq:der_V_1}
    \end{align}
    Now let us bound the right-hand side of this inequality. From \cref{lemm:1}\ref{it:lemm:1-2}, we have:
    \[
        \Vert \Err(x(t))\Vert = \Vert \Err(x(t)) - \Err(\bar{x})\Vert \leq \frac{\gamma\beta\tau}{\lambda}\Vert x(t)-\bar{x}\Vert,
    \]
    and from \cref{lemm:2}\ref{it:lemm2-2}:
    \[
        \Vert\partial_t\Err(x(t)) \Vert \leq \frac{c_0}{\lambda}\|\dot{x}(t)\|,
    \]
    where $c_0 = 2+ 2\gamma L + \gamma \beta \tau$. Since $\V(t)\geq \frac{1}{2}\Vert v(t)\Vert^2$, applying Young's inequality to the first product on the right-hand side yields:
    \begin{equation}
        \begin{aligned}
            \Vert v(t)\Vert\Vert\Err(x(t))+\omega\partial_t\Err(x(t))\Vert 
            &\leq \frac{\sqrt{\tmu}}{16}\Vert v(t)\Vert^2 + \frac{8}{\sqrt{\tmu}}\Vert\Err(x(t))\Vert^2+\frac{8\omega^2}{\sqrt{\tmu}}\Vert\partial_t \Err(x(t))\Vert^2\\ 
            &\leq \frac{\sqrt{\tmu}}{8}\V(t) + \frac{8\gamma^2\beta^2\tau^2}{\sqrt{\tmu}\lambda^2}\Vert x(t) - \bar{x}\Vert^2+\frac{8\omega^2 c_{0}^{2}}{\sqrt{\tmu}\lambda^2}\Vert \dot{x}(t)\Vert^2.
        \end{aligned}
    \end{equation}
    Plugging this back, \eqref{eq:der_V_1} becomes:
    \begin{equation}\label{eq:der_V_2}
        \dot{\V}(t)+\frac{\sqrt{\tmu}}{8}\V(t) + \BF{a}\Vert\dot{x}(t)\Vert^2 + \BF{b}\Vert x(t) - \bar{x}\Vert^2 \leq \delta\Vert \partial_t \T (x(t))\Vert \Vert x(t) - \xbar\Vert,
    \end{equation}
    with $\BF{a} \coloneqq \frac{\sqrt{\tmu}}{2} - \frac{8\omega^2 c_{0}^{2}}{\sqrt{\tmu}\lambda^2}$ and $\BF{b} \coloneqq \frac{\tmu\sqrt{\tmu}}{4} - \frac{8\gamma^2\beta^2\tau^2}{\sqrt{\tmu}\lambda^2}$. 
    We verify that $\BF{a} > 0$ under the assumption $\omega < \frac{\lambda\sqrt{\tmu}}{4 c_0}$ (implied by \eqref{eq:cond_omega}), and $\BF{b} > 0$ under the assumption $\rho = \frac{\beta\tau}{\tmu} < \frac{\sqrt{2}\lambda}{8\gamma}$ from \eqref{eq:cond_rho}.
		
    As for the term $\Vert \partial_t \T (x(t))\Vert \Vert x(t) - \xbar\Vert$, we have by \cref{lemm:T}\ref{it:lemT_3}:
    \[
        \Vert\partial_t\T(x(t))\Vert \leq \frac{4}{\lambda}\Vert\dot{x}(t)\Vert.
    \]
    Using Young's inequality parameterized by some $\epsilon > 0$, we have:
    \[
        \delta\Vert \partial_t \T (x(t))\Vert \Vert x(t) - \xbar\Vert \leq \frac{4\delta}{\lambda}\Vert\dot{x}(t)\Vert\Vert x(t) - \xbar\Vert \leq \frac{2\delta\epsilon}{\lambda}\Vert\dot{x}(t)\Vert^2 + \frac{2\delta}{\epsilon\lambda} \Vert x(t) - \xbar\Vert^2.
    \]
    We get after rearranging all terms:
    \begin{equation}\label{eq:V2}
        \dot{\V}(t) + \frac{\sqrt{\tmu}}{8}\V(t) + \left(\BF{a}-\frac{2\delta\epsilon}{\lambda}\right)\Vert\dot{x}(t)\Vert^2 + \left(\BF{b} - \frac{2\delta}{\epsilon\lambda}\right) \Vert x(t) - \xbar\Vert^2 \leq 0.
    \end{equation}
    We see that $\BF{a} - \frac{2\delta\epsilon}{\lambda} > 0$ requires $\epsilon < \frac{\lambda\BF{a}}{2\delta}$, and $\BF{b} - \frac{2\delta}{\epsilon\lambda} > 0$ requires $\epsilon > \frac{2\delta}{\lambda\BF{b}}$. Consequently, such an $\epsilon$ exists if and only if the interval $I \coloneqq (\frac{2\delta}{\lambda\BF{b}}, \frac{\lambda\BF{a}}{2\delta})$ is non-empty, which is guaranteed under our conditions \eqref{eq:conditions_globales}.
    Choosing $\epsilon \in I$, the quadratic terms are positive and can be discarded, hence:
    \[
        \dot{\V}(t) + \frac{\sqrt{\tmu}}{8}\V(t) \leq 0.
    \]
    Integrating this inequality from $t_0$ to $t$ gives:
    \begin{equation}\label{eq:estimate_V}
        \V(t)\leq \V(t_0) e^{-\frac{\sqrt{\tmu}}{8}(t-t_0)}.
    \end{equation}
    Therefore, $\lim_{t\to\infty} \V(t) = 0$, and in particular,
    \begin{equation}\label{eq:est1}
        \lim_{t\to\infty} \langle \T(x(t)),x(t) - \xbar\rangle = 0 \quad \text{and} \quad \lim_{t\to\infty}\Vert v(t)\Vert = 0.
    \end{equation}
    By the definition of $\V$ and the $\tmu$-strong monotonicity of $\T$ we get:
    \begin{equation}\label{eq:est2}
        \tmu \|x(t)-\bar{x}\|^2 \leq \langle \T(x(t)),x(t) - \xbar\rangle \leq \V(t_0) e^{-\frac{\sqrt{\tmu}}{8}(t-t_0)} \quad \text{and} \quad \Vert v(t)\Vert^2 \leq 2\V(t_0) e^{-\frac{\sqrt{\tmu}}{8}(t-t_0)}.
    \end{equation}
    Moreover, by the $\lambda/2$-cocoercivity of $\T$, we get:
    \[
        \frac{\lambda}{2}\Vert \T(x(t))\Vert^2 \leq \langle \T(x(t)),x(t) - \xbar\rangle \leq \V(t_0) e^{-\frac{\sqrt{\tmu}}{8}(t-t_0)}.
    \]
    Again, coming back to \eqref{eq:est2}, we have:
    \[
        \Vert v(t)\Vert \leq C e^{-\frac{\sqrt{\tmu}}{16}t}
    \]
    with $C = \sqrt{2\V(t_0)} e^{\frac{\sqrt{\tmu}}{16}t_0}$. Since $\dot{x}(t) = v(t) - \sqrt{\tmu}(x(t)-\xbar)$, the triangle inequality gives:
    \[
        \Vert\dot{x}(t)\Vert \leq \Vert v(t) \Vert + \sqrt{\tmu}\Vert x(t) - \xbar\Vert \leq C_1 e^{-\frac{\sqrt{\tmu}}{16}t}.
    \]
    To prove the integral estimate \cref{it:hd-it2}, we use that by cocoercivity of $\T$ combined with \eqref{eq:est2}:
    \[
        \Vert \T(x(t))\Vert^2 \leq \frac{4}{\lambda^2} \Vert x(t) - \xbar\Vert^2 \leq C' e^{-\frac{\sqrt{\tmu}}{8}t},
    \]
    with $C' = \frac{4}{\tmu\lambda^2} \V(t_0)e^{\frac{\sqrt{\tmu}}{8}t_0}$. Integrating this against the exponential weight gives:
    \[
        \int_{t_0}^{t}e^{\sqrt{\tmu}s}\Vert \T(x(s))\Vert^2\dd s \leq C'\int_{t_0}^{t} e^{\frac{7\sqrt{\tmu}}{8}s} \dd s \leq \frac{8}{7\sqrt{\tmu}} C' e^{\frac{7\sqrt{\tmu}}{8}t}.
    \]
    Multiplying by $e^{-\sqrt{\tmu}t}$ finally yields:
    \[
        e^{-\sqrt{\tmu} t}\int_{t_0}^{t} e^{\sqrt{\tmu}s}\Vert\T(x(s))\Vert^2 \dd s \leq C_2 e^{-\frac{\sqrt{\tmu}}{8}t},
    \]
    as desired.
\end{proof}

Similar to \cref{cor:measure_cv1}, we end this section by providing the exponential rate of convergence of the decision-dependent measure $\mm_{x(t)}$ to the equilibrium measure $\mm_{\bar{x}}$ in the $\Wass_1$ distance. This is a direct consequence of \cref{assu:W1} and the trajectory decay established in \cref{thm:convergence_sm}\ref{it:hd-it1}.

\begin{corollary}\label{cor:W1_rate2}
    Let $x: [t_0, \infty) \to \cH$ be the trajectory generated by \eqref{eq:dynamic-sm} under the conditions of \cref{thm:convergence_sm}. Then, for all $t \geq t_0$:
    \[
        \Wass_{1}(\mm_{x(t)},\mm_{\bar{x}}) \leq \tau\sqrt{\frac{\V(t_0)}{\tmu}} e^{-\frac{\sqrt{\tmu}}{16} (t-t_0)}.
    \]
\end{corollary}

\section*{Acknowledgments}

The work of H.E. was supported by the FMJH Program PGMO (grant no. P-2024-0019).
\appendix
\section*{Appendices}

\section{ Fixed point theorems}

						\begin{theorem}[see, \eg \cite{ABG,brezisFA}]\label{thm:fixedpoint}
				Let $(\X,\dd)$ be a complete metric space and $S:\X\to\X$ be a strict contraction, \ie there exists a constant $\rho<1$ such that
				\[
				\dd(S(x),S(y))\leq \rho \dd(x,y), \forall x,y\in\X.
				\]
				Then, there exists a unique $\bar{x}\in\X$ such that $S(\bar{x}) = \bar{x}$. Moreover, for any $x_0\in\X$, the sequence starting from $x_0$ with $x_{n+1} = S(x_n)$ for all $n\in\nats$ converges to $\bar{x}$ as $n$ goes to $\infty$.
			\end{theorem}
			\begin{theorem}[see, \eg \cite{Boyd&Wong}]\label{thm:boyd-wong}
    Let $(\X,\dd)$ be a complete metric space and $S:\X\to\X$ be a nonlinear contraction, \ie there exists a continuous function $\psi: [0, \infty) \to [0, \infty)$ satisfying $\psi(0) = 0$ and $\psi(t) < t$ for all $t > 0$, such that
    \[
        \dd(S(x),S(y))\leq \psi(\dd(x,y)), \forall x,y\in\X.
    \]
    Then, there exists a unique $\bar{x}\in\X$ such that $S(\bar{x}) = \bar{x}$. Moreover, for any $x_0\in\X$, the sequence starting from $x_0$ with $x_{n+1} = S(x_n)$ for all $n\in\nats$ converges to $\bar{x}$ as $n$ goes to $\infty$.
\end{theorem}

\bibliographystyle{plainurl}

\bibliography{smi_splt}

\end{document}